\documentclass[12pt,reqno,a4paper]{amsart}
\usepackage{hyperref}
\usepackage{latexsym,amsmath}
\usepackage{enumerate}
\usepackage{amsfonts}
\usepackage{amssymb}
\usepackage{geometry}
\usepackage{latexsym}
\usepackage{fixmath}
\usepackage[T1]{fontenc}
\usepackage{fourier}
\usepackage{bbm}
\usepackage{color}
\usepackage{xcolor}
\usepackage{fullpage}
\usepackage{cleveref}
\usepackage{ulem}

\usepackage{tikz}
\usetikzlibrary{calc, shapes, backgrounds}
\usepackage{./fjimacros}

\newcommand\nl{\nonumber\\&}
\newcommand\longdto{\overset{d}{\longrightarrow}}

\newcommand{\Prob}{{\mathbb{P}}}

\newcommand{\N}{{\mathbb N}}

\DeclareMathOperator{\FIND}{FIND}
\newcommand\swap{^{\mathrm{swap}}}
\newcommand\comp{} 

\newtheorem{definition}{Definition}[section]

\newtheorem{remark}[definition]{Remark}
\newtheorem{theorem}[definition]{Theorem}
\newtheorem{lemma}[definition]{Lemma}
\newtheorem{proposition}[definition]{Proposition}
\newtheorem{corollary}[definition]{Corollary}

\usepackage{tikz}
\usetikzlibrary{arrows}
\usepackage{float}

\title{On fine fluctuations of the complexity of the QuickSelect algorithm}
\author[J. Ischebeck \and R. Neininger]{Jasper Ischebeck, Ralph Neininger}
\thanks{}

\address{Jasper Ischebeck, {\tt ischebec@math.uni-frankfurt.de}, Goethe University Frankfurt, Institute of Mathematics, 60629 Frankfurt a.M., Germany }
\address{Ralph Neininger, {\tt neininger@math.uni-frankfurt.de}, Goethe University Frankfurt, Institute of Mathematics, 60629 Frankfurt a.M., Germany }

\begin{document}

\begin{abstract}
The Quickselect algorithm (also called FIND) is a fundamental algorithm for selecting ranks or quantiles within a set of data.
 Grübel and Rösler showed that the number of key comparisons required by Quickselect considered as a process of the quantiles $\alpha\in[0,1]$  converges within a natural probabilistic model after normalization in distribution within the càdlàg space $D[0,1]$ endowed with the Skorokhod metric.

We show that the residual process in the latter convergence after normalization converges in distribution towards a mixture of Gaussian processes in $D[0,1]$. A similar result holds for the related algorithm QuickVal.
Our method is applicable to other cost measures such as the number of swaps (key exchanges) required by Quickselect, or cost measures being based on key comparisons taking additionally into account that the cost of a comparison between two keys may depend on their values, an example being the number of bit comparisons needed to compare keys given by their bit expansions. For all the arising mixtures of Gaussian limit processes, we also discuss the H\"older continuity of their paths.
\end{abstract}
\maketitle

\section{Introduction and results}

In 1961, Hoare \cite{hoareQuickselect} introduced the Quickselect algorithm, which he called FIND, to select a key (an element) of a given rank from a linearly ordered finite set of data. We assume that the data are distinct real numbers. To be definite, a simple version of the Quickselect algorithm is given as follows: $\FIND(S,k)$ has as input a set $S=\{s_1,\ldots,s_n\}$ of distinct real numbers of size $n$ and an integer $1\le k \le n$. The algorithm $\FIND$ operates recursively as follows:
If $n=1$ we have $k=1$ and $\FIND$ returns the single element of $S$. If $n\ge 2$ and $S=\{s_1,\ldots,s_n\}$ the algorithm first chooses an element from $S$, say $s_j$, called pivot, and generates the sets
\begin{align*}
    S_<:=\{s_i\,|\, s_i<s_j, i\in\{1,\ldots,n\}\setminus\{j\}\}, \quad S_\ge:=\{s_i\,|\, s_i\ge s_j, i\in\{1,\ldots,n\}\setminus\{j\}\}.
\end{align*}
If $k=|S_<|+1$, the algorithm returns $s_j$. If $k \le |S_<|$, recursively  $\FIND(S_<,k)$ is applied. If $k \ge |S_<|+2$,  recursively  $\FIND(S_\ge,k-|S_<|-1)$ is applied. Note that $\FIND(S,k)$ returns the element of rank $k$ from $S$. There are various variants of the algorithm, in particular regarding how the pivot element is chosen and how $S$ is partitioned into the subsets $S_<$ and $S_\ge$.

In a standard probabilistic model one assumes that the data are ordered, i.e.~given as a vector $(s_1,\ldots,s_n)$, and are randomly permuted, all permutations being equally likely. This can be achieved assuming that the data are given as $(U_1,\ldots,U_n)$ where $(U_j)_{j\in \NN}$ is a sequence of i.i.d.~random variables with distribution $\mathrm{unif}[0,1]$, the uniform distribution over the unit interval $[0,1]$. This is the probabilistic model considered below. Note that the randomness is within the data, while the algorithm is deterministic.

Various cost measures have been considered for Quickselect, mainly the number of key comparisons required, which we also analyse. Furthermore, we analyse the number of swaps (key exchanges) required and cost measures which are based on key comparisons, where the cost of a comparison may depend on the values of the keys $s_i$, $s_j$, and the number of bit comparisons required to decide whether $s_i<s_j$ is a prominent example.


For analysis purposes a related process, called  $\mathrm{QuickVal}$, has been considered, see \cite{vaclfifl09, FillNakama}. Informally, $\mathrm{QuickVal}$ for an $\alpha\in[0,1]$ mimics QuickSelect to select (or to try to select) the value $\alpha$ from the set of data, which, in our probabilistic model for large $n$, comes close to QuickSelect selecting rank $\lfloor \alpha n\rfloor$.  To be definite, $\mathrm{QuickVal}((U_1,\ldots,U_n),\alpha)$ compares the $U_i$ with $U_1$ to generate sublists
\begin{align*}
    S_<:=(U_{j_1},\ldots,U_{j_{m-1}}), \quad S_\ge:=(U_{j_{m+1}},\ldots,U_{j_n}),
\end{align*}
with $U_{j_i}< U_1$ for $i=1,\ldots,{m-1}$ and $2\le j_1<\cdots<j_{m-1}$ and  $U_{j_i}\ge U_1$ for $i=m+1,\ldots,n$ and $2\le j_{m+1}<\cdots<j_n$. Hence, $m-1\in\{0,\ldots,n-1\}$ is the number of the $U_i$, $2\le i\le n$, being smaller than $U_1$. The algorithm recursively calls  $\mathrm{QuickVal}(S_<,\alpha)$  if $\alpha < U_1$ and $|S_<|>0$. If $\alpha \ge U_1$ and $|S_\ge|>0$ recursively $\mathrm{QuickVal}(S_\ge,\alpha-U_1)$ is called. The number of key comparisons required by $\mathrm{QuickVal}((U_1,\ldots,U_n),\alpha)$ is denoted by $S_{\alpha,n}$.

To describe the processes $(S_{\alpha,n})_{\alpha\in[0,1]}$ and their limit (after scaling) conveniently we also consider the binary search tree constructed from the the data $(U_i)_{i\in\N}$. Part of the following definitions are depicted in Figure \ref{fig_1}.  The data are inserted into the rooted infinite binary tree, where we denote its nodes by the elements of $\mg{0,1}^\ast:=\cup_{n=0}^\infty \{0,1\}^n$ as follows. Its root is denoted by the empty word $\epsilon$ and for each node $\phi\in\mg{0,1}^\ast$ we denote by $\phi0$ and  $\phi1$ (the word $\phi$ appended with a $0$ resp.~$1$) its left and right child respectively. Moreover, $\abs\phi$ denotes the length of the word $\phi$, which is the depth of the corresponding node in the tree. To construct the binary search tree for $(U_1,\ldots,U_n)$ the first key $U_1$ is inserted into the root and occupies the root. Then,
successively the following keys are inserted, where each key traverses the already occupied nodes starting at the
root as follows: Whenever the key traversing is less than the occupying key at a node it moves on to the left
child of that node, otherwise to its right child. The first empty node found is occupied by the key.

To describe the costs of the algorithms we organize, using notation of Fill and Nakama \cite{FillNakama}, the sub-intervals $([L_\phi,R_\phi))_{\phi\in \mg{0,1}^\ast}$ implicitly generated starting with $[0,1)=:[L_\epsilon,R_\epsilon)$
and recursively setting
\begin{align}
    \tau_\phi &:= \inf\mg{i\in\NN \,|\,L_\phi < U_i < R_\phi},
    \nonumber\\
    L_{\phi0} &:= L_{\phi}, \quad
    R_{\phi1} := R_{\phi},\quad
    L_{\phi1} := R_{\phi0} := U_{\tau_\phi},\quad
    I_{\phi} := R_{\phi}-L_\phi.
    \label{defi:irl}
\end{align}
Now, if a sublist starting with pivot $U_{\tau_\phi}$ has to be split by $\mathrm{QuickVal}$, the keys which are inserted in the subtree rooted at $U_{\tau_\phi}$ need to be compared with $U_{\tau_\phi}$. Hence, we get a contribution of key comparisons of
\begin{equation}
    S_{\phi,n} = \sum_{\tau_\phi<k≤n}\eins_{[L_\phi,R_\phi)}(U_k)
    \label{defi:snphi}.
\end{equation}

Now, for $\alpha\in[0,1]$, $\mathrm{QuickVal}((U_1,\ldots,U_n),\alpha)$ generates and splits sublists encoded by $\phi(\alpha,k)$ for $k=0,1,\ldots$ for which we obtain by  $\phi(\alpha,0)=\epsilon$  and
\begin{equation}
    \phi(\alpha,k+1) = \begin{cases}
        \phi(\alpha,k)0, & \mbox{if } \alpha < U_{\tau_{\phi(\alpha,k)}}, \\
        \phi(\alpha,k)1, & \mbox{if } \alpha ≥ U_{\tau_{\phi(\alpha,k)}}.
    \end{cases}
\end{equation}
When using the variables defined in \eqref{defi:irl} or \eqref{defi:snphi}, we abbreviate the notation $\phi(\alpha,k)$ by $\alpha,k$, such as writing $I_{\alpha,k} := I_{\phi(\alpha,k)}$ or $S_{\alpha,k,n} := S_{\phi(\alpha,k),n}$.

\begin{figure}
    \centering
    \tikzset{
        nicht/.style = {top color=white, bottom color=gray!40,
            label = center:\textsf{}},
        normal/.style = {minimum size = 12mm, top color=white, bottom color=blue!20,
            label = center:\textsf{\Large }},
        gross/.style = {minimum size = 12mm, top color=white, bottom color=gray!40, text = black,
            label = center:\textsf{}}
    }
    \begin{tikzpicture}[
    scale = 1, transform shape, thick,
    every node/.style = {draw, circle, minimum size = 8mm},
    grow = down,  
    level 1/.style = {sibling distance=4cm},
    level 2/.style = {sibling distance=2cm},
    level 3/.style = {sibling distance=2cm},
    level distance = 1.5cm
    ]
    \node[normal] (Start) { $U_{\tau_{\epsilon}}$}
    child {   node [gross] (A) {$U_{\tau_0}$}}
    child {   node [normal] (D) {$U_{\tau_{\alpha,1}}$}
        child { node [normal] (E) {$U_{\tau_{\alpha,2}}$}
            child{ node [normal] (G) {...} }}
        child { node [gross] (F) {$U_{\tau_{10}}$}}
    };

    \begin{scope}[nodes = {draw = none}]
    \path (Start) -- (D) node [near start, right] {$\alpha≥U_{\tau_{\epsilon}}$};
    \path (D)     -- (E) node [near start, left]  {$\alpha<U_{\tau_{\alpha,1}}$};
    \end{scope}

    \end{tikzpicture}
    \caption{Part of the binary search tree. The pivots of sublists split by $\mathrm{QuickVal}((U_1,\ldots,U_n),\alpha)$ for some $\alpha\in[0,1]$ are on the path indicated. Note that we have $\tau_\epsilon=\tau_{\phi(\alpha,0)}=\tau_{\alpha,0}=1$ and  in this example $\alpha\ge U_1$ and $\alpha<U_{\tau_{\alpha,1}}$ so that $\phi(\alpha,2)=10\in\{0,1\}^2$. }
 \label{fig_1}   \end{figure}
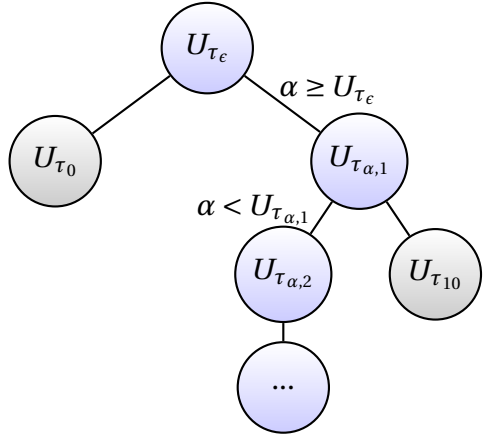
The number of key comparisons required by $\mathrm{QuickVal}((U_1,\ldots,U_n),\alpha)$ is then given by the (finite) sum
\begin{align*}
S_{\alpha,n} = \sum_ {k=1}^{\infty}S_{\alpha,k,n}.
\end{align*}
%
%
Fill and Nakama \cite[Theorem 3.2]{FillNakama} showed (considering more general complexity measures) that for each $\alpha\in[0,1]$ almost surely
\begin{align}\label{limit_fn}
\frac{1}{n}S_{\alpha,n} \to S_\alpha := \sum_{k=0}^{\infty}I_{\alpha,k}, \quad  (n\to\infty).
\end{align}
The latter convergence also holds in $L_p$, see Fill and Matterer \cite[Proposition 6.1]{fima14}.

We take the point of view that such an almost sure asymptotic result may be considered a strong law of large numbers (SLLN). The subject of the  present paper is to study the fluctuations in such SLLN,  sometimes called a central limit analogue.  We study these fluctuations as processes in the metric space $(D[0,1],d_\mathrm{SK})$ of càdlàg functions endowed with the Skorokhod metric; see \Cref{sec:d01} for the definitions and Billingsley \cite{billing} for background on weak convergence of probability measures on metric spaces in general and on $(D[0,1],d_\mathrm{SK})$ in particular. Note that, by definition,  $(S_{\alpha,n})_{\alpha\in[0,1]}$ and  $(S_{\alpha})_{\alpha\in[0,1]}$ have càdlàg paths almost surely.

The analysis of $\mathrm{QuickVal}$ is usually considered an intermediate step to analyze the original $\FIND$ algorithm. Gr\"ubel and R\"osler \cite{roesler:quickselect} already pointed out that a version of  $\FIND$ such as stated above with $C^\ast_n(k)$ denoting the number of key comparisons for finding rank $k$ within $(U_1,\ldots,U_n)$ does not lead  to convergence within $(D[0,1], d_\mathrm{SK})$ after the normalization $\alpha \mapsto \frac{1}{n}C^\ast_n(\floor{\alpha n}+1)$, where here and below the convention $C^\ast_n(n+1):=C^\ast_n(n)$ is used. To overcome this problem they propose a version that does not stop in case a pivot turns out to be the rank to be selected by including the pivot in the list $S_<$ and proceeding until a list of size $1$ is generated. Moreover, their pivots are chosen uniformly at random. The number of key comparisons  $C'_n(k)$ for Grübel and Rösler's $\FIND$-version has the property that
\begin{align}\label{GR_lilaw}
    \kl*{\frac{1}{n}C'_n\kl[\big]{\floor{\alpha n}+1}}_{\alpha\in[0,1]} \longdto (S_\alpha)_{\alpha\in[0,1]}, \quad \mbox{ in } (D[0,1],d_\mathrm{SK}),
\end{align}
see \cite[Theorem 4]{roesler:quickselect}. We may also obtain right-continuous limits with our deterministic choice of pivots by just recursively calling $\FIND(S_\ge,0)$ in case the pivot turns out to be the rank sought. We denote the number of key comparisons for this version by $C_n(k)$, which is close to Grübel and Rösler's $\FIND$-version and also satisfies \eqref{GR_lilaw}.

The convergence in \eqref{GR_lilaw} could only be stated weakly (not almost surely) since Gr\"ubel und R\"osler's $\FIND$-version due to randomization within the algorithm does not have a natural embedding on a probability space. Note that the formulation of the $\mathrm{QuickVal}$ complexity does have such an embedding which, e.g., makes the almost sure convergence in \eqref{limit_fn} possible. Also, it is easy to see that  we have the  distributional equality
\begin{equation}\label{qs:to:qv}
  \kl[\big]{C_{n}\kl[\big]{\abs{\mg{U_i≤\alpha:1≤i≤n}}}}_{\alpha\in[0,1]} \overset{d}{=}   (S_{\alpha,n})_{\alpha\in[0,1]}.
\end{equation}
This allows to naturally couple the complexities on one probability space, which we call its \textit{natural coupling}. See \cite[page 807]{fima14} for a related discussion of natural couplings.

\subsection{Results}
In this section, we collect our results for three different complexity measures, the number of key comparisons, the number of key exchanges (also called swaps), and measures being based on key comparisons taking additionally into account that the cost of a comparison between two keys may depend on their values, an example being the number of bit comparisons needed to compare keys given by their bit expansions.

Part of our results has been announced in the extended abstract \cite{ischne04}.
\subsubsection{Number of key comparisons}
As the normalized process of fluctuations we  define
\begin{align}\label{fluc_pro}
 G_n:=(G_{\alpha,n})_{\alpha\in[0,1]} := \left(\frac{S_{\alpha,n} - nS_\alpha}{\sqrt{n}}\right)_{\mathrlap{\alpha\in[0,1]}}.
\end{align}
Then we have the following result:
\begin{theorem}\label{mainTheorem}
    Let $S_{\alpha,n}$ be the number of key comparisons required by  $\mathrm{QuickVal}((U_1,\ldots,U_n),\alpha)$ and $(S_{\alpha})_{\alpha\in[0,1]}$ as in \eqref{limit_fn}. Then for the fluctuation process $G_n$ defined in  \eqref{fluc_pro} we have
\begin{align*}
    G_n \longdto G_\infty \; \mbox{ in } (D[0,1],d_\mathrm{SK})\qquad (n\to\infty),
\end{align*}
where $G_\infty$ is a mixture of centred Gaussian processes with random
covariance function given by
\begin{equation}\label{def:sigma:inf}
    \Sigma_{\infty, \alpha, \beta} := \sum_{k=0}^{J}\sum_{j=0}^\infty I_{\alpha, j \maxv k} + \mathbf{1}_{\{\alpha\neq \beta\}} (J+1)\sum_{j=J+1}^\infty (I_{\beta, j}) - S_\alpha S_\beta, \quad \alpha,\beta\in[0,1],
\end{equation}
where $J = J(\alpha,\beta):=\max\{k\in\N_0\,|\,
\tau_{\alpha,k} = \tau_{\beta,k}\}\in\N_0\cup\{\infty\}$. We have the representation
\begin{equation}
    G_{\infty} \overset d= \sum_{\phi\in\mg{0,1}^\ast}Z_\phi \eins_{[L_\phi,R_\phi)},
\end{equation}
with $Z_\phi$ defined in Section \ref{sec:lim_rep}.
\end{theorem}
\begin{remark}
An alternative representation of the random covariance function in \eqref{def:sigma:inf} is as follows: With a random variable $V$ uniformly distributed over $[0,1]$ and  independent of the $(U_i)_{i\in\N}$, we have
\begin{align}\label{cov_desc}
    \Sigma_{\infty, \alpha, \beta} = \Cov\kl[\big]{J(V,\alpha), J(V,\beta) \:\big\vert\: \mathcal F_\infty},
\end{align}
with the $\sigma$-algebra
\begin{align}\label{def_f_inf}
 \mathcal F_\infty:=\sigma\mg[\big]{I_\phi\mid \phi\in\mg{0,1}^\ast}.
\end{align}
\end{remark}
\begin{remark} A related functional limit law for the complexity of Radix Selection, an algorithm to select ranks based on the bit expansions of the data, with a limiting Gaussian process with a covariance function related to \eqref{cov_desc} can be found in \cite[Theorem 1.2]{lenesu19}. See \cite[Theorem 1.1]{drnesu14} for another related functional limit law.
\end{remark}
\begin{remark}
In the recent preprint Fill and Matterer \cite[Theorems 5.1 and 6.1.]{fill:quickselect-residual}  convergence of the one-dimensional marginals for the functional limit law in Theorem \ref{mainTheorem} is shown, in distribution and with all moments. See also the PhD thesis, Matterer \cite[Theorem~6.4]{Matterer}.
\end{remark}

To transfer Theorem \ref{mainTheorem} to  $\FIND$ we need to align jumps to come up with a suitable fluctuation process. The conventions $C_n(0):=C_n(1)$ and  $C_n(n+1):=C_n(n)$ are used.




\begin{corollary}\label{cor:qs}
    Let $C_n(k)$ be the number of key comparisons required to select rank $1\le k\le n$ within a set of $n$ data by $\FIND$ with the natural coupling \eqref{qs:to:qv}.
    Let $\Lambda_n:[0,1] \to [0,1]$, $n\in\NN$, be any (random)  monotone increasing bijective function such that
    $\Lambda_n\kl{\frac k{n+1}}$ is equal to the element of rank $k$ within  $\{U_1,\dots, U_n\}$. Then we have
    \begin{align*}
        \left(\frac{C_n(\floor{t (n+1)}) - nS_{\Lambda_n(t)}}{\sqrt{n}}\right)_{t\in[0,1]}   \longdto G_\infty \; \mbox{ in } (D[0,1],d_\mathrm{SK}),
    \end{align*}
 where $G_\infty$ is the process defined in Theorem \ref{mainTheorem}.
\end{corollary}

\subsubsection{Number of swaps} \label{subsec_swap}
Usually, QuickSelect is implemented in-place, which means that it only requires the memory for the list $S$ of values and a bounded amount of additional memory. This is achieved by generating the sets $S_<$ and $S_\ge$ by swapping values within the ordered list that contains the elements of $S$ so that the elements of $S_<$ and $S_\ge$ are moved to contiguous parts of the original list. Such a procedure is called \textit{partition}. There are various \textit{partition} procedures, we are discussing two of them.

The original \textit{partition} procedure by Hoare \cite{hoarePartition} searches the list $S$ from both ends at once: It repeatedly finds the index $i = \min\mg{2≤i≤n\mid U_i>U_1}$ of the leftmost element bigger than the pivot and the index $j:=\max\mg{2≤j≤n\mid U_j<U_1}$ of the rightmost element smaller than the pivot. If $i<j$, it swaps $U_i$ and $U_j$. Else the algorithm terminates.

A simpler, but less efficient implementation is the so-called Lomuto partition scheme  \cite{bentley00,coleri,mahmoud_swaps} that only searches from one end of $S$. It keeps track of the amount $i$ of elements at the start of the list it has already swapped. In every step, it finds the index $j:=\max\mg{2≤j≤n\mid U_j<U_1}$ of the rightmost element smaller than the pivot. If $i+1<j$, it swaps $U_{i+1}$ and $U_{j}$ and increases $i$ by one. Otherwise, the algorithm terminates.

Both partition schemes only compare elements to the pivot, so the model of randomness is preserved within the sublists $S_<$ and $S_\ge$. However, their original order is not preserved, so QuickSelect run on $U_1, \dots, U_n$ will usually not select the same pivots as QuickSelect on $U_1,\dots, U_{n+1}$. For convenience, we assume that the pivot to split a sublist $S'$ of $S$ is the element of $S'$ that came first in the original list $S$. We call this choice of the pivots a \textit{suitable embedding}. \\

\noindent
\textit{ (i) Hoare's partition.}
For Hoare's partition procedure, via a hypergeometric distribution, for details see Section \ref{Sec:swaps}, the expected number of swaps in step $k$ given $\mathcal F_\infty$ is approximately $nI_{\alpha,k+1}(I_{\alpha}-I_{\alpha,k}) / I_{\alpha,k}$, which leads to the  limit process $L=(L_\alpha)_{\alpha\in[0,1]}$ given by
\begin{equation}\label{hoare:limit}
    L_{\alpha} := \sum_{k=0}^\infty \frac{I_{\alpha,k+1}(I_{\alpha}-I_{\alpha,k})}{I_{\alpha,k}},\quad \alpha\in[0,1].
\end{equation}
With the  $Z_\phi$ appearing in Theorem \ref{mainTheorem} and $\{Y_\phi\,|\, \phi\in\mg{0,1}^\ast\}$ a set of i.i.d.\ $\mathcal{N}(0,1)$ random variables being independent of $\{Z_\phi\,|\, \phi\in\mg{0,1}^\ast\}$ and of $\mathcal F_\infty$ the limiting process $G\swap=(G\swap_{\alpha})_{\alpha\in[0,1]}$ is given by
\begin{equation}\label{defi:gswap}
    G\swap_{\alpha} := \;\sum_{\mathclap{\phi\in\{\phi(k,\alpha)\,|\, k\in\N_0\}}}\quad
      Y_\phi\frac{I_{\phi0}I_{\phi1}}{I_{\phi}^{3/2}}
    + Z_{\phi0}\frac{I_{\phi1}}{I_{\phi}}
    + Z_{\phi1}\frac{I_{\phi0}}{I_{\phi}}
    - Z_{\phi}\frac{I_{\phi0}I_{\phi1}}{I_{\phi}^2}, \quad \alpha\in[0,1].
\end{equation}
The $Y_\phi$ represent fluctuations caused by the hypergeometric distribution related to the pre\-sent par\-ti\-tion procedure, while the terms containing $Z_\phi$ represent fluctuations around the limit \eqref{hoare:limit}.
Then we have the following result for key exchanges corresponding to Theorem \ref{mainTheorem}:
\begin{theorem}\label{satz:swap}
    Let $K_{\alpha,n}$ be the number of key exchanges required by $\mathrm{QuickVal}((U_1, \dots, U_n),\allowbreak\alpha)$ with Hoare's partition algorithm in a suitable embedding. Then, as $n\to\infty$, we have
\begin{align*}
    \left(\frac{K_{\alpha,n}-nL_{\alpha}}{\sqrt n}\right)_{\alpha\in[0,1]} \longdto G\swap \quad \mbox{in } (D[0,1],d_\mathrm{SK}).
 \end{align*}
    Furthermore, the latter convergence also holds jointly with the convergence for the comparisons,
\begin{align*}
\left(\frac{S_{\alpha,n} - nS_\alpha}{\sqrt{n}}\right)_{\alpha\in[0,1]}\longdto G_{\infty}  \quad \mbox{in } (D[0,1],d_\mathrm{SK}),
\end{align*}
from \Cref{mainTheorem}, when $G\swap$ and $G_\infty$ are constructed as in \eqref{defi:gswap} and \eqref{ginf-explizit} respectively using the same $Z_\phi$.
\end{theorem}

\noindent
\textit{ (ii) Lomuto's partition.}

The Lomuto partition procedure is simpler to implement and much easier to analyze. It swaps the pivot and every element smaller than the pivot, so the amount of swaps at some path $\phi\in\mg{0,1}^*$ is given by $S_{\phi0}+1$. Thus, the limiting process $ G^\mathrm{Lo}= (G^\mathrm{Lo}_{\alpha})_{\alpha\in[0,1]}$ is
 \begin{equation*}
G^\mathrm{Lo}_{\alpha}=\sum_{k=0}^\infty Z_{\phi(\alpha,k)0},\quad \alpha\in [0,1].
 \end{equation*}
Our proof for the number of key comparisons in Theorem \ref{mainTheorem} can be straightforwardly transferred to yield:
\begin{theorem}\label{thm:lomuto}
    Let $K_{\alpha,n}^\mathrm{Lo}$ be the number of key exchanges required by  $\mathrm{QuickVal}((U_1, \dots, U_n),\alpha)$ with Lomuto's partition procedure in a suitable embedding. Then, as $n\to\infty$, we have
  \begin{equation}
    \kl*{
        \frac{K_{\alpha,n}^\mathrm{Lo}
        -n\sum_{k=0}^\infty I_{\phi(\alpha,k)0}}{\sqrt{n}}
    }_{\alpha\in[0,1]}
    \longdto G^\mathrm{Lo}
    \quad \mbox{in } (D[0,1],d_\mathrm{SK}).
 \end{equation}
\end{theorem}

\subsubsection{$\eps$-tame cost functions} \label{subsec:bits}
We now consider the model where the cost to compare two keys depends on their values.  These costs are described by a measurable \textit{cost function} $\beta: [0,1]^2 \to [0,\infty)$, and we require that they have a polynomial tail, that is: There are constants $c$, $\eps>0$ such that for all $u\in[0,1], x \in \NN$ and for $V$ being $\unif[0,1]$ distributed
\begin{equation*}
    \Pk*{\beta(u,V)≥x} \le c x^{-1/\eps}.
\end{equation*}
This condition is called $(c,\eps)$\textit{-tameness}, see \cite{FillNakama, Matterer, fill:quickselect-residual}, and  $\beta$ is called $\eps$\textit{-tame} if it is $(c,\eps)$-tame for some $c>0$. Note that, e.g., $\beta$ counting the number of bit comparisons is   $\eps$-tame for all $\eps>0$.
The costs of $\mathrm{QuickVal}((U_1,\ldots,U_n),\alpha)$ in this model are given by
\begin{equation*}
    S^\beta_{\alpha,n} := \sum_{k=0}^\infty \sum_{\tau_{\alpha,k}<i≤n} \eins_{[L_{\alpha,k},R_{\alpha,k})}(U_i) \beta(U_{\tau_{\alpha,k}},U_i)
\end{equation*}
and the limit is, with $V$ being $\unif[0,1]$ distributed and independent of the $U_1,\dots, U_n$, given as
\begin{equation*}
    S^\beta_{\alpha,\infty} := \sum_{k=0}^\infty \Ek*{\eins_{[L_{\alpha,k},R_{\alpha,k})}(V) \beta(U_{\tau_{\alpha,k}},V) \given \mathcal F_\infty}
\end{equation*}
\cite{fill:quickselect-residual} show for $\eps < \frac12$ that for fixed $\alpha\in[0,1]$ the resulting residual
\begin{equation*}
    G^\beta_{\alpha,n} := \frac{S^\beta_{\alpha,n}-nS^\beta_{\alpha,\infty}}{\sqrt n}
\end{equation*}
converges to a mixed centred Gaussian random variable $G^\beta_{\alpha,\infty}$ in distribution and with all moments. It is possible to combine them to a mixture of centred Gaussian processes
\begin{align}\label{mix_gau_g_beta}
G^\beta_{\infty}=(G^\beta_{\alpha,\infty})_{\alpha\in[0,1]},
\end{align}
defined by the conditional covariance functions given, with
\begin{equation}\label{def:xbeta}
    X^\beta_{\alpha,k} :=
    \eins_{[L_{\alpha,k},R_{\alpha,k})}(V)\cdot\beta(U_{\tau_{\alpha,k,n}},V),
\end{equation}
by
\begin{equation}\label{cov:beta}
    \Cov\kl[\big]{G^\beta_{\alpha,\infty},
    G^\beta_{\gamma,\infty} \mid \mathcal F_\infty}
    = \Cov\kl*{\sum_{k=0}^\infty X^\beta_{\alpha,k},
    \sum_{k=0}^\infty X^\beta_{\gamma,k}
    \:\big\vert\: \mathcal F_\infty},\quad \alpha,\gamma \in [0,1].
\end{equation}
We show later, in \Cref{lem:k1}, that the latter covariance is well-defined for $s<\eps\inv$.

Note that now the total cost for all key comparisons required by $\mathrm{QuickVal}((U_1,\ldots,U_n),\alpha)$ or  $\FIND((U_1,\ldots,U_n),k)$ is no longer determined by the fact that the ranks of $(U_1,\ldots,U_n)$ form an uniformly random permutation. Here, the distribution of the $U_i$ matters. We only consider the uniform distribution as for the other cost measures in the present paper and leave other distributions for future work. We have the following result corresponding to Theorem \ref{mainTheorem}.
\begin{theorem}[$\eps$-tame cost function] \label{thm:beta}
   Let $S^\beta_{\alpha,n}$ be the cost of $\mathrm{QuickVal}((U_1, \dots, U_n),\alpha)$ for an  $\eps$-tame cost function $\beta$ with $\eps<\frac14$. Then we have
    \begin{align*}
      \left(  \frac{S^\beta_{\alpha,n}-nS^\beta_{\alpha,\infty}}{\sqrt n}\right)_{\alpha\in[0,1]} \longdto  G^\beta_\infty  \quad \mbox{ in } (D[0,1],d_\mathrm{SK}),
    \end{align*}
where $G^\beta_\infty$ is the mixture of centred Gaussian processes defined in  \eqref{mix_gau_g_beta}.
\end{theorem}

\subsubsection{Path properties of the limit processes}\label{sec:path}
It is obvious from the covariance function that the limit process $G_\infty$ is a.s.\ discontinuous at every $U_i$, $i\in\NN$, within the Euclidean metric. We equip $[0,1]$ with two equivalent (random) metrics $d_2$ and $d_G$, which are $\mathcal F_\infty$-measurable, and show that $G_\infty$ is Hölder continuous in these metrics.

The first metric we consider is
\begin{equation}\label{defi:d2}
    d_2(\alpha,\beta) := 2^{-J(\alpha,\beta)}
\end{equation}
with $J(\alpha,\beta)$ defined in Theorem \ref{mainTheorem}. Then,  $d_2$ is an ultrametric (i.e. the longest two sides of a triangle have the same length) and the balls in this metric are the intervals $[L_{\phi}, R_\phi)$ resp. $[L_\phi,R_\phi]$ in the case of $R_\phi=1$. While we could use another base than 2, this choice  is the natural choice in which the Hausdorff dimension of $([0,1],d_2)$ is one. A function defined on $([0,1],d_2)$ is continuous if and only if it is càdlàg and continuous in the Euclidean metric at all points outside the set $\mg{U_\phi\mid\phi\in\mg{0,1}^\ast}$. Consequently, the topology induced by $d_2$ is coarser than the Euclidean topology on $[0,1]$.

The second metric we consider is the canonical metric of the (mixed) Gaussian process $G\comp_\infty$ given by
\begin{equation}\label{def:dg}
    d_G(\alpha,\beta) := \Ek[\big]{\kl{G_{\alpha,\infty}\comp-G\comp_{\beta,\infty}}^2\given\mathcal F_\infty}^{\frac12}.
\end{equation}
We show in \Cref{lem:d2dg-equivalent} that the metrics $d_2$ and $d_G$ are a.s.~equivalent.

\begin{theorem}\label{satz:hölder}
For the limit processes $G_\infty\comp$, $G_\infty\swap$ and  $G^\beta_\infty$ appearing in Theorems \ref{mainTheorem}, \ref{satz:swap} and \ref{thm:beta} we have:
\begin{enumerate}
    \item The processes $G_\infty\comp$ and $G_\infty\swap$ are a.s.\ Hölder continuous with respect to $d_2$ for any Hölder exponent $\gamma < \frac14\log_2\frac32\approx 0.1462$ and Hölder continuous with regard to $d_G$ for any Hölder exponent $\gamma<1$.
    \item For an $\eps$-tame comparison cost function $\beta$ with $\eps<1/4$, the residual process $G^\beta_\infty$ is a.s.\ $d_2$-Hölder continuous with Hölder exponent $\gamma < (1-2\eps)\frac14\log_2\frac32$ and $d_G$-Hölder continuous with Hölder exponent $\gamma < 1-2\eps$.
\end{enumerate}\end{theorem}

\section{Proofs}
The present section is organized as follows: Section \ref{sec:d01} contains an abstract criterion, Proposition \ref{lem:convergence}, for weak convergence of probability measures on $(D[0,1],d_\mathrm{SK})$. All the functional limit theorems in the present paper are obtained via Proposition \ref{lem:convergence}. To verify the conditions of Proposition \ref{lem:convergence} first in Section \ref{Sec:pert} a novel perturbation argument is introduced, which is the basis of our analysis. Then in Section \ref{Sec:proof} estimates are presented to use Proposition \ref{lem:convergence} to prove the main result, Theorem \ref{mainTheorem}, on the number of key comparisons. Section \ref{sec:lim_rep} is an intermezzo where the explicit representation of the limit process appearing in Theorem \ref{mainTheorem} is added. Then, Sections \ref{Sec:swaps} and \ref{subsec:proof:bits} contain the estimates analogous to Section \ref{Sec:proof} to prove the main results on the number of swaps and on $\eps$-tame cost functions from Theorems \ref{satz:swap} and \ref{thm:beta} respectively. The final Section \ref{sec:paths} contains the proof of the path properties of the limit processes stated in Theorem \ref{satz:hölder}.


\subsection{On Weak Convergence in \texorpdfstring{$D[0,1]$}{D[0,1]}} \label{sec:d01}

Within the space $D[0,1]$ of all càdlàg functions, closeness of functions $f,g \in D[0,1]$ is measured in the  Skorokhod metric
\begin{align}
    d_{\mathrm{SK}}(f,g) = \inf_\lambda \max \mg[\big]{\absnorm{f\circ\lambda-g},\,\absnorm{\lambda-\id}},
\end{align}
where the infimum is taken over all increasing bijections $\lambda:\:[0,1]\to[0,1]$ and  $\id$ denotes identity.

To obtain weak convergence on $(D[0,1],d_\mathrm{SK})$ we first show tightness so that our sequences of processes in Theorems  \ref{mainTheorem}, \ref{satz:swap} and \ref{thm:beta} have  weakly convergent subsequences respectively, and then show uniqueness of the limit. The tightness on $(D[0,1], d_{\mathrm{SK}})$ can be shown using an equivalent of the Arzelà--Ascoli theorem, whereas the uniqueness follows from the fdd convergence.

For some $x\in D[0,1]$ and $\delta>0$, we define a version of a  modulus of continuity by
\begin{equation}\label{def:modulus}
    w'_x(\delta) = \inf_{\mg{t_i}}\max_{0<i≤r}\sup_{t_{i-1}≤s,u<t_{i}}\abs{x(s)-x(u)}.
\end{equation}
The infimum is taken over finite sets $\mg{t_i}_{i=1}^r$ with
\begin{equation*}
    \begin{cases}
        0=t_0< \dots< t_r = 1 \text{ and}      \\
        t_i - t_{i-1} > \delta \quad \text{ for } i=1,2,\dots, r.
    \end{cases}
\end{equation*}
We start recalling a well-known tightness criterion:
\begin{theorem}[Theorem 13.2 in \cite{billing}]\label{lem:tightness_appendix}
    Let $\mg{P_n}_{\ninn}$ be a sequence of probability measures on $(D[0,1],d_\mathrm{SK})$. The family $\mg{P_n}$ is tight if and only if the following two conditions are satisfied: \begin{enumerate}
        \item For each $\eta>0$, there exists an $a>0$ such that
        \begin{equation}\label{maximum-tight}
           \sup_{n\ge 1} P_n(\mg{x:\: \absnorm{x} > a}) ≤ \eta,
        \end{equation}
        \item For all $\eps, \eta>0$, there exist $0<\delta<1$ and $n_0\in\NN$ such that
        \begin{equation}\label{modulus-tight}
          \sup_{n\ge n_0}  P_n(\mg{x:\: w'_x(\delta)≥\eps})≤\eta.
        \end{equation}
    \end{enumerate}
\end{theorem}
We derive from the latter theorem our own criteria for convergence:

\begin{proposition}[Convergence Criterion]\label{lem:convergence}
    Let $X_1,X_2,\dots$ be a sequence of random variables in $(D[0,1],d_\mathrm{SK})$. Suppose that for every $K\in\NN$, there exist random càdlàg step functions $X^K_1,X^K_2,\dots$ such that their jumps are all contained in the set $\mg{U_\phi\mid \abs\phi< K}$. If \begin{enumerate}
        \item For all $\alpha_1,\dots, \alpha_r \in [0,1]$, the marginals $\mathcal L\kl{X_n(\alpha_1),\dots, X_n(\alpha_n)}$ converge weakly to some distribution $\mu_{\alpha_1,\dots, \alpha_r}$, and
        \item For all $\eps>0$, \begin{equation*}
            \lim_{K\to\infty}\limsup_{n\to\infty} \Pk[\Big]{\absnorm[\big]{X_n-X^K_{n}}>\eps} = 0,
        \end{equation*}
    \end{enumerate}
    then $X_n$ converges weakly to some random process $X$ on $(D[0,1],d_\mathrm{SK})$ as $n\to \infty$, and for $\alpha_1,\dots, \alpha_r \in [0,1]$, we have
    \begin{equation}\label{lem:convergence:marginals}
        (X\kl{\alpha_1},\dots, X\kl{\alpha_r}) \sim \mu_{\alpha_1,\dots, \alpha_r}.
    \end{equation}
\end{proposition}

\begin{proof}
    We first show that the set of distributions $\mg{\P_{X_1},\P_{X_2},\dots}$ is tight using \Cref{lem:tightness_appendix}.
    Observe that since $I_\phi>0$ a.s.\ for all $\phi\in\mg{0,1}^*$, there exists for every $K$ a random variable $\delta_K$ such that all points in $\mg{U_\phi\mid \abs\phi <K }$ are at least $\delta_K$ away from 0, 1 and each other.
    Hence, for every $\delta≤\delta_K$, the points $\mg{0,1} \cup \mg{U_\phi\mid \abs\phi <K}$ are a valid choice for the grid $\mg{t_0<\dots<t_r}$ in \eqref{def:modulus}. But since the processes $X^K_n$ are constant between these points,
    \begin{equation*}
        w'_{X_n}(\delta) ≤ \max_{0<i≤r}\sup_{t_{i-1}≤s,u<t_{i}}\abs{X_n(s)-X_n(u)} ≤ 2\absnorm[\big]{X_n-X^K_n}.
    \end{equation*}
    To show \eqref{modulus-tight}, let $\eps>0$. Then, for every $n,K\in\NN$ and $\delta>0$,
    \begin{equation*}
        \Pk{w'_{X_n}(\delta) > \eps} ≤ \Pk{\delta_K < \delta} + \Pk[\Big]{\absnorm{X_{n}-X_{n}^K} > \frac{\eps}2}.
    \end{equation*}
    For $n$ and $K$ sufficiently large, the last term is arbitrarily small; by then choosing $\delta$ small, the first term is also small, showing \eqref{modulus-tight}. We now can use \eqref{modulus-tight} to show \eqref{maximum-tight}: Let $r\in\NN$. Every subinterval with length at least $2^{-r}$ contains a dyadic number $j2^{-r}$, $j\in\NN$. Therefore,
    \begin{equation*}
        \absnorm{X_n} ≤ w'_{X_n}(2^{-r}) + \max_{j=0,\ldots,2^r} \:\abs[\big]{X_n\kl{j2^{-r}}}.
    \end{equation*}
    By \eqref{modulus-tight}, the first term fulfills
    \begin{equation}\label{pf:conv:1}
        \Pk[\big]{w'_{X_n}(2^{-r})>1}<\eta/2
    \end{equation} for sufficiently large $r$ and $n$.
    By our first condition, $\kl{X_n\kl{j2^{-r}}\mid 0≤j≤2^r}$ converges weakly for $n\to\infty$, so it is a tight family on $\RR^{2^r}$. Therefore there exists a $\rho\in \RR$ such that
    \begin{equation}\label{pf:conv:2}
        \Pk*{\max_{j=0,\ldots,2^r} \:\abs[\big]{X_n\kl{j2^{-r}}} > \rho} < \eta/2.
    \end{equation}
Combining \eqref{pf:conv:1} and \eqref{pf:conv:2}, we get \eqref{maximum-tight} for all $n$ bigger than some $n_0$, by possibly further enlarging $\rho$, this also holds for $n<n_0$.

    By \Cref{lem:tightness_appendix} $\mg{\P_{X_1},\P_{X_2},\dots}$ is tight, so $X_1,X_2,\dots$ has a subsequence converging in distribution to some random variable $X$. By Theorem 13.1 in \cite{billing} and the discussion preceding it, condition (1) implies that $X$ is the unique limit and that
    \begin{equation*}
        X\kl{\alpha_1},\dots, X\kl{\alpha_r} \sim \mu_{\alpha_1,\dots, \alpha_r}
    \end{equation*}
    for $\alpha_1,\dots, \alpha_r \in [0,1]$ such that $X$ is a.s.\ continuous in $\alpha_1,\dots, \alpha_r$. For \eqref{lem:convergence:marginals}, it remains to show that for all $a\in[0,1]$ the function $X$ is a.s.\ continuous in $a$.

    On $D[0,1]$, the function $x\mapsto x(a)-x(a-)$ has discontinuities, but we may approximate it with
    \begin{equation*}
        J_{\delta,a}(x) := \frac1{\delta} \int_{0}^\delta x(a+s)\mathrm ds - \frac1{\delta}\int_{-\delta}^0 x(a+s)\mathrm ds \longrightarrow x(a)-x(a-) \quad (\delta\downarrow 0).
    \end{equation*}
  Since Skorokhod convergence implies pointwise convergence almost everywhere and convergence of the maximum, we obtain that $J_{\delta,a}$ is a continuous function.

    We show that $J_{\delta,a}(X) \to 0$ for $\delta\to0$ a.s. Let $\eps>0$, then by weak convergence
    \begin{align*}
        \lim_{\delta\to 0} \Pk*{J_{\delta,a}(X) > \eps}
        ≤ \lim_{\delta\to 0} \lim_{n\to\infty} \Pk*{J_{\delta,a}(X_n) > \eps}.
    \end{align*}
    For any $K\in\NN$, we have $J_{\delta,a} ≤ \absnorm{X_n-X^K_n}$ if there is no $U_\phi$, $\abs\phi<K$ closer than $\delta$ to $a$, so
    \begin{align*}
        \Pk*{J_{\delta,a}(X_n) > \eps}
        ≤ \inf_{K\in\NN}
        \Pk*{\absnorm{X_n-X^K_n} > \eps} + \Pk*{\exists\phi:\: \abs\phi<K,\,\abs{U_\phi-a}<\delta}.
    \end{align*}
    The first term is small for $n$, $K$ sufficiently large, the second term is small for $\delta$ sufficiently small. With this, $X$ is a.s.\ continuous in $a$.

\end{proof}

\subsection{Perturbation of the data} \label{Sec:pert}

$\mathrm{QuickVal}$ splits an interval $[L_{\phi},R_{\phi})$ by the first value falling into  $[L_{\phi},R_{\phi})$ denoted by $U_{\tau_{\phi}}$. Obviously, this implies dependencies between the data $U_i$ and the lengths $I_\phi$ of the intervals $[L_{\phi},R_{\phi})$. In the present section we construct a perturbed sequence $(\widetilde U_i)_{i\in\NN}$ to the data $(U_i)_{i\in\NN}$ such that we gain independence of $(\widetilde U_i)_{i\in\NN}$ from the $\sigma$-algebra $ \mathcal F_\infty$ of the  interval lengths defined in  \eqref{def_f_inf}. In particular, we aim that conditional on  $\mathcal F_\infty$ the number of data $(\widetilde U_1,\ldots,\widetilde U_n)$ falling into an interval $[L_{\phi},R_{\phi})$ is binomial $B(n, I_{\phi})$ distributed, see Lemma \ref{lem:stilde} below.

Every value $U_i$, $i\in\NN$, falls successively into subintervals generated by  $\mathrm{QuickVal}$ until becoming a pivot element. These subintervals correspond to the path between the root of the corresponding binary search tree and the node  where $U_i$ is inserted. Let $\phi_i \in\mg{0,1}^\ast$ denote the node where $U_i$ is inserted. Hence, we have $\tau_{\phi_i}=i$ and $U_i = L_{\phi_i} + I_{\phi_i0}$.

Let $(V_i)_{i\in\NN}$ be a sequence of i.i.d.~$\mathrm{unif}[0,1]$ random variables being independent of $(U_i)_{i\in\NN}$. We define
\begin{equation}\label{def_u_tilde}
\widetilde U_i := L_{\phi_i} + I_{\phi_i}V_i.
\end{equation}

\begin{lemma}\label{lem_pert}
The sequence $(\widetilde U_i)_{i\in\NN}$ defined in \eqref{def_u_tilde} consists of i.i.d.~$\mathrm{unif}[0,1]$ distributed random variables and is independent of $\mathcal F_\infty$.
\end{lemma}
\begin{proof}
It suffices to show that $\widetilde U_i$ conditional on  $\mathcal F_\infty$ and $\widetilde U_1,\dots, \widetilde U_{i-1}$ is uniformly distributed on $[0,1]$ for all $i\in\NN$. We use infinitesimal notation to denote this claim by
\begin{equation*}
    \Pk*{\widetilde U_i \in \mathrm du\given \mathcal F_\infty, \widetilde U_1,\dots, \widetilde U_{i-1}}
    = \mathbf{1}_{[0,1]}(u) \mathrm du,\quad i\in\NN.
\end{equation*}
For each $i\in \NN$ the random variables $\widetilde U_i$ and $U_i$ fall into the same interval $[L_{\phi_i}, R_{\phi_i})$, hence  $\phi_1,\ldots,\phi_{i-1}$ are determined by $\widetilde U_1,\dots,\widetilde U_{i-1}$. Let us additionally condition on $\phi_i$, then, by definition,
\begin{align*}
    \Pk*{\widetilde U_i \in \mathrm du\given \mathcal F_\infty, \widetilde U_1,\dots, \widetilde U_{i-1}, \phi_i} = \frac{1}{I_{\phi_i}}\eins_{[L_{\phi_i}, R_{\phi_i})}(u)\mathrm du.
\end{align*}
Note that $\phi_i$ denotes one of the $i$ external nodes of the binary search tree with internal nodes denoted by $\phi_1,\ldots,\phi_{i-1}$. We denote by $\mathrm{Ext}_{i-1}$ the set of the labels of these external nodes. Hence, conditional on $\mathcal F_\infty, \phi_1,\dots, \phi_{i-1}$ the label $\phi_i$ is chosen from $\mathrm{Ext}_{i-1}$ with probability given by the length of the corresponding interval, i.e., $\Pk{\phi_i=\phi\given \mathcal F_\infty, \phi_1,\dots, \phi_{i-1}}=I_\phi$ for all $\phi\in \mathrm{Ext}_{i-1}$. Thus, by the law of total probability we obtain
\begin{equation*}
    \Pk*{\widetilde U_i \in \mathrm du\given \mathcal F_\infty, \widetilde U_1,\dots, \widetilde U_{i-1}}
    = \sum_{\phi\in \mathrm{Ext}_{i-1}} I_{\phi}\frac{1}{I_{\phi}}\eins_{[L_\phi,R_\phi)}(u)\mathrm du = \mathbf{1}_{[0,1]}(u)\mathrm du.
\end{equation*}
This implies the assertion.
\end{proof}
The $\widetilde U_i$ are now coupled with the $U_i$ but independent of the $I_\phi$. To compare with the number of key comparisons required by $\mathrm{QuickVal}((U_1,\ldots,U_n),\alpha)$ we define
\begin{equation*}
    \widetilde S_{\alpha,k,n} := \sum_{i=1}^n \eins_{[L_{\alpha,k},R_{\alpha,k})}(\widetilde U_i).
\end{equation*}
\begin{lemma}\label{lem:stilde}
    Conditional on $I_{\alpha,k}$ we have that $\widetilde S_{\alpha,k,n}$ has the binomial B$(n,I_{\alpha,k})$ distribution. Moreover, for all $\alpha\in[0,1]$, $n\in \NN$ and $0\le k\le n$ we have
    \begin{equation}
        S_{\alpha,k,n}\le (\widetilde S_{\alpha,k,n} - 1)^+ ≤ S_{\alpha,k,n} + k - 1.
    \end{equation}
\end{lemma}
\begin{proof}
    The conditional distribution of $\widetilde S_{\alpha,k,n}$ follows from Lemma \ref{lem_pert}.
    Recall that $S_{\alpha,k,n}$ is defined as $\sum_{i=\tau_{\alpha,k}}^n \eins\mg{L_{\alpha,k-1}\le U_i < R_{\alpha,k-1}}$.  By definition, $U_i$ and $\widetilde U_i$ are in the interval $[L_{\phi_i}, R_{\phi_i})$ for all $i\in\NN$. If $U_i\in(L_{\alpha,k}, R_{\alpha,k})$, then $U_i$ appears as a pivot after the $k$-th pivot. Hence, its interval $[L_{\phi_i}, R_{\phi_i})$ and thus also $\widetilde U_i$ are contained in $(L_{\alpha,k}, R_{\alpha,k})$. The $k$-th pivot $U_{\tau_{\alpha,k}}$ itself does not contribute to $S_{\alpha,k,n}$, which implies the left inequality stated in the present lemma.

    For the right inequality, assume for some $i\in\NN$ that the perturbed value $\widetilde U_i$ is in $(L_{\alpha,k}, R_{\alpha,k})$, but $U_i$ is not. Then the corresponding interval $(L_{\phi_i}, R_{\phi_i})$ must
    contain $(L_{\alpha,k}, R_{\alpha,k})$, thus making $U_i$ a pivot that appears before the $k$-th pivot. Since there are only $k$ such pivots, the right inequality follows.
\end{proof}


\subsection{Proof of Theorem \ref{mainTheorem}} \label{Sec:proof}

To split the contributions to the process $G_n$ into costs resulting from above and below a level $K\in\NN$ we define
\begin{align}\label{def:gakn}
    G_{\alpha,k,n} := \frac{S_{\alpha,k,n} - nI_{\alpha,k}}{\sqrt n}
\end{align}
as the normalized fluctuations of the contribution at level $k$, and set
\begin{align}
    G^{\le K}_{\alpha,n} := \sum_{k=0}^K  G_{\alpha,k,n}, \qquad
    G^{\le K}_{n}:=\kl[\big]{G^{\le K}_{\alpha,n}}_{\alpha\in[0,1]},\qquad
    G^{>K}_{\alpha,n} := \sum_{k=K+1}^\infty G_{\alpha,k,n}.
    \label{def:glek}
\end{align}
Hence, $G_{\alpha,n} =  G^{\le K}_{\alpha,n} + G^{>K}_{\alpha,n}$. Analogously, for the perturbed values $\widetilde S_{k,n}$ we define
\begin{align}
    W_{\alpha,k,n} := \frac{\widetilde S_{\alpha,k,n} -nI_{\alpha,k}}{\sqrt n}, \qquad
    W^{\le K}_{\alpha,n} := \sum_{k=0}^K  W_{\alpha,k,n} \qquad
    W^{\le K}_{n} :=  \kl[\big]{W^{\le K}_{\alpha,n}}_{\alpha\in[0,1]}.
    \label{def:w}
\end{align}

\begin{lemma}\label{lem:clt}
For all  $K\in \NN$ we have convergence in distribution of $(G^{\le K}_n)_{n\in\NN}$ towards a mixture $G^{\le K}_\infty =(G^{\le K}_{\alpha,\infty})_{\alpha\in[0,1]}$ of centred Gaussian processes within $\|\cdot\|_\infty$. Conditional on $\mathcal F_\infty$, the limit  $G^{\le K}_\infty$ is a centred Gaussian  process with covariance function given, for  $\alpha,\beta\in[0,1]$ by
\begin{align}\label{def:sigma:k}
         \Cov\kl[\big]{G^{\le K}_{\alpha,\infty}, G^{\le K}_{\beta,\infty} \mid \mathcal F_\infty}
        = \sum_{k=0}^{K \minv J}\sum_{j=0}^K I_{\alpha, j \maxv k} + \kl[\big]{1 + (K \minv J)}\sum_{j=J+1}^K I_{\beta, j} - S_\alpha^{\le K} S_\beta^{\le K},
    \end{align}
    where $J = J(\alpha,\beta)$ is as in Theorem \ref{mainTheorem} and $S_\alpha^{\le K} := \sum_{k=0}^K I_{\alpha,k}$.
    The stated convergence in distribution also holds conditionally in $\mathcal F_\infty$, i.e., we have almost surely that  $\mathcal L(G^{\le K}_n\mid \mathcal F_\infty)$  converges weakly towards $\mathcal L(G^{\le K}_\infty\mid \mathcal F_\infty)$.
\end{lemma}

\begin{proof}
    First note that by Lemma \ref{lem:stilde} we have $\absnorm*{G^{\le K}_n - W^{\le K}_n} < K^2/\sqrt{n}$, so it suffices to show the lemma for $W^{\le K}_n$. Conditional on $\mathcal F_\infty$, the value of $W^{\le K}_{\alpha,n}$ is given by
    \begin{equation}
        W^{\le K}_{\alpha,n} = \frac1{\sqrt n}\sum_{i=1}^n\sum_{k=0}^{K} \eins\mg{L_{\alpha,k} ≤ \widetilde U_i < R_{\alpha,k}} - (R_{\alpha,k}-L_{\alpha, k}),
    \end{equation}
    thus the $2^k$ different values of the process $W^{\le K}_n$ can be expressed as the sum of $n$ centred, bounded i.i.d.\ random vectors, scaled by $1/\sqrt{n}$. By the multivariate central limit theorem, these converge towards a multivariate, centred normal variable. As the positions of the jumps, still conditional on $\mathcal F_\infty$, are fixed, we have convergence of $W^{\le K}_n$ and thus also of $G^{\le K}_n$ towards a Gaussian process.
    Define $X_{\alpha,k} := \eins\mg{L_{\alpha,k-1} ≤ \widetilde U_1 < R_{\alpha,k-1}}$. The covariance function then is given by
    \begin{align}\label{cov_oben}
        \Cov\kl[\big]{G^{\le K}_{\alpha,\infty}, G^{\le K}_{\beta,\infty} \mid \mathcal F_\infty}
        &= \Cov\kl[\bigg]{
            \sum_{k=0}^{K}X_{\alpha,k},
            \sum_{j=0}^{K}X_{\beta,j}
            \:\bigg\vert\: \mathcal F_\infty
        }
        \nl = \sum_{k=0}^{K}\sum_{j=0}^{K}\Ek*{X_{\alpha,k}X_{\beta,j} \given \mathcal F_\infty} - S^{\le K}_\alpha S^{\le K}_\beta
        \nl = \sum_{k=0}^{K \minv J}\sum_{j=0}^K I_{\alpha, k \maxv j} + \kl[\big]{1+\kl{K \minv J}}\sum_{j=J+1}^K I_{\beta, j} - S_\alpha^{\le K} S_\beta^{\le K}.
    \end{align}
    The assertion follows.
\end{proof}


To see that the covariance functions in \eqref{cov_oben} converge towards the  covariance function of  $G_\infty$ stated in Theorem \ref{mainTheorem} we restate a Lemma of Gr\"ubel and R\"osler \cite[Lemma 1]{roesler:quickselect} that the maximal length of the intervals at a level is decreasing geometrically with increasing levels. It is obtained observing that $\Ek[\big]{\sum_{\abs\phi=k}I_\phi^2} = \fracc23^k$ and states:
\begin{lemma}\label{lem:k1}
    There exists an a.s.\ finite random variable $K_1$ such that for all $k\ge K_1$:
    \begin{equation}\label{interval-size}
        \max_{\alpha \in [0,1]} I_{\alpha, k} \le k\fracc23^{k/2}.
    \end{equation}
\end{lemma}

\Cref{lem:k1} implies that the covariance functions of $G^{≤K}_\infty$ from \eqref{def:sigma:k} converge a.s.\ to the covariance function of $G_\infty$ from \eqref{def:sigma:inf}.

For the costs from levels $k>K$ we find:

\begin{proposition}\label{prop:rest-klein}
    For all $\eps,\eta>0$ there are constants
    $K,N \in \NN$ such that for all $n≥N$
    \begin{equation}
        \Pk[\big]{\absnorm{G^{>K}_n} > \eta} < \eps.
    \end{equation}
\end{proposition}
We postpone the proof of the above \Cref{prop:rest-klein} and first use it and \Cref{lem:clt} to show convergence of the finite-dimensional distributions, denoted fdd-convergence.
\begin{lemma}\label{lem:fdd} We have fdd-convergence of $G_n$ towards $G_\infty$.
\end{lemma}
\begin{proof}
    For any $K$, we can split $G_n=G^{≤K}_n + G^{>K}_n$. By \Cref{lem:clt}, we have
    \begin{align*}
    G^{≤K}_n \overset{\mathrm{fdd}}{\longrightarrow} G^{≤K}_\infty \qquad (n\to\infty).
    \end{align*}
    Furthermore, because the covariance functions of the  $G^{≤K}_\infty$ converge a.s., we obtain
        \begin{align*}
    G^{≤K}_\infty \overset{\mathrm{fdd}}{\longrightarrow}\ G_\infty \qquad (K\to\infty)
     \end{align*}
by Lévy's continuity theorem. Hence, for all $\alpha_1,\dots, \alpha_\ell\in [0,1]$ and all $t_1, \dots, t_\ell\in\RR$ we find a sequence $(K_n)_{n\in\NN}$ in $\NN$ such that
    \[
        \Pk[\big]{G^{≤K_n}_{\alpha_1,n} < t_1, \dots, G^{≤K_n}_{\alpha_\ell,n} < t_\ell} \longrightarrow \Pk[\big]{G_{\alpha_1,\infty} < t_1, \dots, G_{\alpha_\ell,\infty} < t_\ell} \quad (n\to\infty).
    \]
Now, since $\absnorm{G^{>K_n}_n} \to 0$ in probability by \Cref{prop:rest-klein} the claim of \Cref{lem:fdd} follows by Slutzky's theorem.
\end{proof}

To prepare for the proof of \Cref{prop:rest-klein}, we show that the fluctuations on each level are also at least geometrically decreasing. Recall $K_1$ from Lemma \ref{lem:k1}.
\begin{lemma}\label{lem:lvlCosts}
    For $a>1$ sufficiently small, there exist constants $b>1$ and $C_b>0$ such that for all $k,n\in\NN$
    \begin{equation*}
        \Pk*{\max_{\alpha \in[0,1]} \abs{W_{\alpha,k,n}}>a^{-k},\, K_1≤k} \le C_b\exp\kl*{-b^k} + c(k,n),
    \end{equation*}
    where
    \begin{equation*}
        c(k,n) := 2^{k+1}\exp\kl*{
            -\frac12a^{-k}\sqrt n
        }
    \end{equation*}
    In particular, we can choose $a$ so small that for $n\to\infty$
    \begin{equation}\label{reihen}
         \sum_{k=1}^{\floor{(9/2)\log n}}c(k,n) \longrightarrow 0.
    \end{equation}
\end{lemma}
For the proof of Lemma \ref{lem:lvlCosts} we require the following Chernoff bound:
\begin{lemma}\label{lem:chernoff}
    Let $S_n$ be binomial $B(n,p)$ distributed for some $p\in[0,1]$ and $n\in\NN$ and let $\mu := \Ek{S_n}$, $\eps≥0$. Then
    \begin{equation*}
        \Pk[\Big]{S_n \notin \kl[\big]{(1-\eps)\mu, (1+\eps)\mu}}
        ≤ 2\exp\kl*{-\frac{\eps^2\mu}{2+\eps}}.
    \end{equation*}
\end{lemma}
\begin{proof}
    Combine upper and lower bound in McDiarmid \cite[Theorem 2.3]{McDiarmid:concentration}.
\end{proof}
\begin{proof}[Proof of \Cref{lem:lvlCosts}]
    Fix some $\alpha\in[0,1]$.
    Conditionally on $I_{k,\alpha}$, the costs $\widetilde S_{\alpha,k,n}$ are $B(n, I_{k,\alpha})$-distributed by Lemma \ref{lem_pert}. The Chernoff bound in \Cref{lem:chernoff} implies
    \begin{align}
        \Pk[\big]{\abs{W_{\alpha,k,n}}>a^{-k} \given I_{\alpha,k}}
        &= \Pk*{
            \abs[\big]{\widetilde S_{\alpha,k,n} - nI_{\alpha,k}} > \sqrt{n}a^{-k}
            \given I_{\alpha,k}
        }
        \nl ≤ 2\exp\kl*{-\frac{
                na^{-2k}
            }{
                nI_{\alpha,k}\kl{2+\sqrt{n} a^{-k}/(nI_{\alpha,k})}}
            }
        \nl = 2\exp\kl*{-\kl*{2a^{2k}I_{\alpha,k} + a^k/\sqrt{n}}\inv}.
        \label{prf:chernoff}
    \end{align}
    Because the function $x\mapsto\exp(-x\inv)$ is monotonically increasing and a sum is smaller than twice the maximum,
    \begin{align}
        \Pk[\big]{\abs{W_{\alpha,k,n}}>a^{-k} \given I_{\alpha,k}}
        &≤ 2\exp\kl*{-\kl*{4a^{2k}I_{\alpha,k} \maxv 2a^k/\sqrt{n}}\inv}
        \nl=
        2\exp\kl*{-\kl*{4a^{2k}I_{\alpha,k}}\inv} \maxv 2\exp\kl*{-\frac12 a^{-k}\sqrt n}.
    \end{align}
    To bound $W_{\alpha,k,n}$ for all $\alpha,k$, we sum up over all $\phi \in \mg{0,1}^k$,
    \begin{align}
        \Pk[\Big]{\max_{\alpha\in[0,1]}\abs{W_{\alpha,k,n}}>a^{-k}\given I_{\alpha,k}}
        &≤ 2^{k+1}\exp\kl*{-\frac14a^{-2k}I_{\alpha,k}\inv} + 2^{k+1}\exp\kl*{-\frac12 a^{-k}\sqrt n}
    \end{align}
    When furthermore $k≥K_1$, we have $I_{\alpha,k}<k\fracc23^{k/2}$,
    so
    \begin{equation}
        \Pk[\Big]{\max_{\alpha\in[0,1]}\abs{W_{\alpha,k,n}}>a^{-k},\, K_1 \le k}
        ≤
        2^{k+1}\exp\kl*{
            -\frac14a^{-2k}k\inv \fracc32^{k/2}
        }
        +
        2^{k+1}\exp\kl*{
            -\frac12a^{-k}\sqrt n
        }.
    \end{equation}
    For $a^4<1.5$, the term $a^{-2k}\fracc32^{k/2}$ is $\Omega(b^k)$ for some $b>1$, and we can choose the constant $C_b>0$ appropriately.
\end{proof}

\noindent We are now prepared for the proof of \Cref{prop:rest-klein}.

\begin{proof}[Proof of \Cref{prop:rest-klein}]
  Let $\varepsilon,\eta>0$.  To $K_1$ from Lemma \ref{lem:k1} and $a, b$ and $C_b$ from Lemma \ref{lem:lvlCosts} we choose $K$ sufficiently large such that
    \begin{align}\label{3bound}
        \Pk{K_1>K} ≤ \frac{\eps}{4},\qquad \sum_{k=K}^\infty a^{-k} ≤ \frac{\eta}{4} \qquad \mbox{and} \qquad \sum_{k=K}^\infty C_b\exp(-b^k) ≤\frac{\eps}{4}.
    \end{align}
    Let $H_n$ be the maximum amount of steps needed by  $\mathrm{QuickVal}((U_1,\ldots,U_n),\alpha)$  for any $\alpha$. Thus, $H_n$ is also the height of the binary search tree built from $U_1,\dots, U_n$.
    Devroye \cite{devroye:bst} showed that the height has expectation $\Ek{H_n} = \gamma \log n + o(\log n)$ with $\gamma = 4.311\ldots$ Reed \cite{reed:bst} and Drmota \cite{dr02,dr03} further showed that $\Var(H_n) = \mathrm{O}(1)$. Hence, we can choose $N$ sufficiently large such that
    \begin{align}\label{2bounds}
        \Pk[\big]{H_n>\floor{(9/2)\log n}} < \frac{\eps}{4},\qquad
        \sum_{k=1}^{\floor{(9/2)\log n}} c(k,n) ≤ \frac{\eps}{4},\qquad\text{and}\qquad
        \frac{\kl*{(9/2)\log n}^2}{\sqrt n} ≤ \frac\eta4
    \end{align}
    for all $n≥N$. Subsequently we use the decomposition
     \begin{align*}
        G^{>K}_{\alpha,n}= \sum_{k=K+1}^{\floor{(9/2)\log n}} \frac{S_{\alpha,k,n}-nI_{\alpha,k}}{\sqrt{n}} +  \sum_{\floor{(9/2)\log n}+1}^\infty \frac{S_{\alpha,k,n}-nI_{\alpha,k}}{\sqrt{n}} =: \Gamma_n +  G^{>\floor{(9/2)\log n}}_{\alpha,n}
    \end{align*}
    and consider the event
    \begin{align*}
       A_n:=\{K_1>K\} \cup \{H_n>\floor{(9/2)\log n}\}.
    \end{align*}
    We have $\Prob(A_n)<\varepsilon/2$ for all $n\ge N$. Note that on $A_n^c$ (the complement of $A_n$) we have $S_{\alpha,k,n}=0$ for all $k>\floor{(9/2)\log n}$ and also the bound on $I_{\alpha,k}$ from Lemma \ref{lem:k1} applies, hence
    \begin{align*}
     \abs*{G^{>\floor{(9/2)\log n}}_{\alpha,n}}
     &\le\sum_{\floor{(9/2)\log n}+1}^\infty\sqrt{n}I_{\alpha,k}
     \le\sum_{\floor{(9/2)\log n}+1}^\infty\sqrt{n}k\fracc23^{k/2}
     \\&=\mathrm{O}\kl*{n^{1/2-(9/4)\log(3/2)}\log n} =\mathrm{o}(1)
    \end{align*}
    since $(9/4)\log(3/2) = 0.912\ldots$. Hence, we can enlarge  $N$ so that on $A_n^c$ we have $\abs[\big]{G^{>\floor{(9/2)\log n}}_{\alpha, n}} < \eta/2$ for all $n\ge N$. This implies the bound
     \begin{align}
        \Pk[\big]{\abs{G^{>K}_{\alpha,n}} > \eta}
         &≤ \Prob(A_n) + \Pk*{\mg*{\abs{G^{>K}_{\alpha,n}} > \eta}\cap A_n^c}
         \nl≤ \frac{\eps}{2} +\Pk*{\mg*{\abs{\Gamma_n} > \frac{\eta}{2}}\cap A_n^c}
         + \Pk*{\mg*{\abs[\big]{G^{>\floor{(9/2)\log n}}_{\alpha,n}}> \frac{\eta}{2}}\cap A_n^c}.
         \label{sum_with3}
    \end{align}
    Note that the third summand in \eqref{sum_with3} is $0$. Hence, it remains to bound the second  summand in \eqref{sum_with3}. To this end note that
    \begin{align}
        \abs{\Gamma_n} &≤
        \sup_{\alpha\in[0,1]}
        \sum_{k=K+1}^{\floor{(9/2)\log n}}
        \kl*{
          \abs*{
            \frac{\widetilde S_{\alpha,k,n}
            -nI_{\alpha,k}}{\sqrt{n}}
          }
          +\abs*{
            \frac{S_{\alpha,k,n}-\widetilde S_{\alpha,k,n}}{\sqrt n}
          }
        }
        \nl ≤ \kl[\Bigg]{\sum_{k=K+1}^{\floor{(9/2)\log n}} \max_{\alpha\in[0,1]}\abs{W_{\alpha,k,n}}} + \frac{\floor{(9/2)\log n}^2}{\sqrt n},
        \label{gamma-n-spaltung}
    \end{align}
    where Lemma \ref{lem:stilde} is used. The third relation in \eqref{2bounds} assures that the second term in \eqref{gamma-n-spaltung} is smaller than $\eta/4$.
    In view of the second relation in \eqref{3bound} and \eqref{gamma-n-spaltung}, we have
    \begin{align*}
       \mg*{\abs*{\Gamma_n} > \frac{\eta}{2}}\cap A_n^c \subset \bigcup_{k=K}^{\floor{(9/2)\log n}}
        \left\{ \max_{\alpha \in[0,1]} \abs{W_{\alpha,k,n}}>a^{-k},\, K_1≤k\right\}.
    \end{align*}
    Thus, Lemma \ref{lem:lvlCosts} together with   \eqref{3bound} and  \eqref{2bounds} imply that the  second  summand in \eqref{sum_with3} is bounded by $\eps/2$. This implies the assertion.
\end{proof}

\begin{proof}[Proof of \Cref{mainTheorem}]
    We apply \Cref{lem:convergence} to $G_n$ and $G_n^{≤K}$. The first condition, fdd convergence, is \Cref{lem:fdd}, the second condition is \Cref{prop:rest-klein}. The representation of the limit process stated in \Cref{mainTheorem} is the subject of Section \ref{sec:lim_rep}.
\end{proof}
We now transfer the fluctuation result for $\mathrm{QuickVal}$ in Theorem \ref{mainTheorem} to the original Quickselect process in \Cref{cor:qs}.
\begin{proof}[Proof of \Cref{cor:qs}]
    Let $\widetilde F_n$ be the inverse of $\Lambda_n$ in the statement of  \Cref{cor:qs}. By definition of $\Lambda_n$, the value of the element $U_{(k)}$ of rank $k$ within $U_1,\dots, U_n$ is given by $\frac k{n+1}$, so
    \begin{equation}\label{cor:qs:fn}
        \floor[\big]{(n+1)\widetilde F_n(\alpha)} = \abs[\big]{\mg*{U_i≤\alpha\mid 1≤ i≤ n}}
    \end{equation}
    for all $\alpha \in [0,1)$. Thus,
    $C_n\kl[\big]{\floor[\big]{(n+1)\widetilde F_n(\alpha)}}=S_{\alpha,n}$ a.s. for all $\alpha\in[0,1]$, see \eqref{qs:to:qv}. For $\alpha=1$ note that $\widetilde F_n(\alpha)=1$ and $C_n(n+1)=C_n(n)$ by definition.
    The Skorokhod distance $d_\mathrm{SK}$ is then bounded by
    \begin{align*}
        d_\mathrm{SK}\kl*{G_{n}, \left(\frac{C_n(\floor{t (n+1)}) - nS_{\Lambda_n(t)}}{\sqrt{n}}\right)_{t\in[0,1]}}
        & = d_\mathrm{SK}\kl*{G_{n}, \left(\frac{S_{\Lambda_n(t),n} - nS_{\Lambda_n(t)}}{\sqrt{n}}\right)_{t\in[0,1]}}
        \nl = d_\mathrm{SK}\kl*{G_{n}, (G_{\Lambda_n(t),n})_{t\in[0,1]}}
        \\
        &≤ \absnorm{\widetilde F_n-\mathrm{id}}.
    \end{align*}
    By \eqref{cor:qs:fn}, $\widetilde F_n$ is close to the empirical distribution function and thus converges a.s.\ uniformly to the identity $\mathrm{id}$ by the Glivenko--Cantelli theorem. The statement of \Cref{cor:qs}  then follows from Slutzky's theorem.
\end{proof}

\subsection{Explicit construction of the limit process}\label{sec:lim_rep}
Note that on a step-wise level, for finitely many $\phi_1,\dots, \phi_m \in\mg{0,1}^\ast$ of paths,
$G_{\phi_1,n},\dots, G_{\phi_m,n}$ converge, conditional on interval sizes, jointly to normally distributed random variables $Z_{\phi_1},\allowbreak \dots, Z_{\phi_m}$, with covariances given by
\begin{equation}\label{z-cov}
    \Ek*{Z_{\phi}Z_{\psi} \given \mathcal F_\infty} = \abs[\big]{[L_\phi,R_\phi) \cap [R_\psi, L_\psi)}
    = \kl[\big]{(R_\phi \minv R_\psi) - (L_\phi \maxv L_\psi)}^+
\end{equation}
for $\phi,\psi \in \mg{0,1}^\ast$, where $\abs{\cdot}$ denotes the length of an interval.
Using the Kolmogorov extension theorem \cite[Theorem~6.17]{kallenberg}, one can construct the entire family 

\begin{lemma}\label{lem:zphi-small}
    Almost surely, for all but finitely many $\phi\in\mg{0,1}^\ast$ we have
    \begin{equation}\label{eq:zphi-small}
        \abs{Z_\phi} < 2\sqrt{\abs{\phi}I_\phi}.
    \end{equation}
\end{lemma}
\begin{proof}
Note that $Z_\phi$ has the normal distribution with variance $I_{\phi} - I_{\phi}^2 > I_{\phi}$. Let $Z$ be $\mathcal{N}(0,1)$-distributed. Then
    \begin{align}
        \Pk*{\abs{Z_\phi} > 2\sqrt{\abs{\phi}I_\phi}}
        ≤ \Pk*{Z > 2\sqrt{\abs{\phi}}}
        ≤ \frac2{\sqrt{2\pi}}\int_{2\sqrt{\abs{\phi}}}^\infty ye^{-y^2/2}\mathrm dy
        = \frac2{\sqrt{2\pi}}e^{-2\abs\phi}.
    \end{align}
    Since $e^2>2$, the sum of these probabilities over all $\phi$ remains finite and hence the Borel--Cantelli lemma implies the assertion.
\end{proof}
With \Cref{lem:zphi-small} and \Cref{lem:k1}, $\sum_{\phi}Z_\phi \eins_{[L_\phi,R_\phi)}$ converges uniformly a.s.\ and
\begin{equation}\label{ginf-explizit}
    G_{\infty} \overset d= \sum_{\phi\in\{0,1\}^\ast}Z_\phi \eins_{[L_\phi,R_\phi)},
\end{equation}
since left- and right-hand sides of the latter display have the same covariance function and thus the same finite dimensional distributions (cf.~\cite[Theorem 13.1]{billing}). We can thus use the left- and right-hand sides of \eqref{ginf-explizit} as an explicit construction of $G_\infty$ and henceforth assume pointwise equality in \eqref{ginf-explizit}.

\subsection{Number of swaps} \label{Sec:swaps}
In this section, we present a proof for Theorem \ref{satz:swap}.

In Hoare's partition algorithm applied to a uniform permutation, the number of swaps is well known to be, conditional on the rank of the pivot element, hypergeometrically distributed, see, e.g., \cite{ma00,mahmoud_swaps,dane}. The parameters of this hypergeometric distribution, in our notation, are $n-1$ (population size), $S_0+1$ (number of trials) and $n-S_0$ (number of successes).
Let $K_{\phi,n}$ denote the number of swaps at the path $\phi\in\mg{0,1}^*$ when searching within $n$ elements. On the $\abs\phi$-th level, the swaps at previous levels may have changed the order of the elements. Let $I_1,\dots, I_{S_{\phi,n}}$ denote the indices of elements that are compared to the pivot $U_\phi$ in the order in which the procedure \textit{partition} sees them. In this notation, the number of swaps at level $k$ is given by
\begin{equation}\label{def:kphi}
    K_{\phi,n} = \abs{\mg{U_{I_i}\mid 1≤i≤S_{\phi0,n},U_{I_i}>U_\phi}}.
\end{equation}

Now we use a variation of the perturbation argument from Section \ref{Sec:pert} as follows. We extend $I$ by the indices $I_{S_{\phi,n}+1},\dots, I_{\widetilde S_{\phi,n}}$ of the elements that only contribute to $\widetilde S_{\phi_n}$ in arbitrary order and let
\begin{equation}\label{def:ktilde}
    \widetilde K_{\phi,n} := \abs*{\mg{\widetilde U_{I_i}\mid 1≤i≤\widetilde S_{\phi0,n},\widetilde U_{I_i}>U_\phi}}.
\end{equation}
By \Cref{lem:stilde}, $S_{\phi0, n} ≤ \widetilde S_{\phi0, n} ≤ S_{\phi0, n} + \abs\phi + 1$ and we therefore sum over at most $\abs{\phi}+1$ more elements compared to $K_{\phi,n}$. Thus, $K_{\phi,n} \le \widetilde K_{\phi,n} ≤ K_{\phi,n} + \abs{\phi}+1$. With \eqref{def:ktilde} also the limit process can be interpreted: We have that $\widetilde S_{\phi0,n}$ is close to $nI_{\phi0}$ and the probability of a value to be bigger than the pivot is $I_{\phi1}/I_{\phi}$, so we expect that $\widetilde K_{\phi,n}$ is close to $n\frac{I_{\phi0}I_{\phi1}}{I_{\phi}}$. The fact that we draw without replacement does asymptotically not matter as $n\to\infty$. Hence, we define the limit process as
\begin{align}
    K_{\phi,\infty} &:= \frac{I_{\phi0}I_{\phi1}}{I_{\phi}},
&
    K_{\alpha,\infty} &:= \sum_{k=1}^\infty K_{\phi(\alpha,k),\infty}.
\end{align}
As for comparisons in \eqref{def:w}, we define a process $\overline W_n$ based on the perturbed $\widetilde K_{\alpha,n}$:
\begin{align}
    \overline W_{\phi,n} &:= \frac{\widetilde K_{\phi,n} -nK_{\phi,\infty}}{\sqrt n},
    &
    \overline W^{\le K}_{\alpha,n} &:= \sum_{k=0}^K  \overline W_{\alpha,k,n},
    &
    \overline W^{\le K}_{n} &:=  \kl[\big]{\overline W^{\le K}_{\alpha,n}}_{\alpha\in[0,1]}.
    \label{def:w:swap}
\end{align}
and the contributions per step for $G\swap$ as
\begin{align}
    \overline G_{\phi,n} &:= \frac{K_{\phi,n} - nK_{\phi,\infty}}{\sqrt n},
    &
    \overline G^{\le K}_{\alpha,n} &:= \sum_{k=0}^K  \overline G_{\phi(\alpha,k),n},
    &
    \overline G^{\le K}_{n} &:=\kl[\big]{\overline G^{\le K}_{\alpha,n}}_{\alpha\in[0,1]}.
\end{align}
Our result for swaps corresponding to \Cref{lem:clt} is as follows:
\begin{lemma}\label{lem:swap:clt}
    For the number of swaps in Hoare's partition procedure, above some level $K\in\NN$ we have, conditional on $\mathcal F_k$, weak convergence in $\absnorm\cdot$ of $\overline G^{\le K}_n$ to the process is given by
    \begin{equation}
        \overline G^{≤K}_{\alpha,\infty} := \sum_{\abs\phi ≤ K} \kl*{
            Y_\phi\frac{I_{\phi0}I_{\phi1}}{I_{\phi}^{3/2}}
            + Z_{\phi0}\frac{I_{\phi1}}{I_{\phi}}
            + Z_{\phi1}\frac{I_{\phi0}}{I_{\phi}}
            - Z_{\phi}\frac{I_{\phi0}I_{\phi1}}{I_{\phi}^2}
        }\eins_{[L_\phi,R_\phi)}(\alpha),\quad \alpha\in[0,1],
    \end{equation}which conditioned on $\mathcal F_k$ is a centred Gaussian process. This convergence also holds unconditionally. The convergence also holds jointly with the convergence of the corresponding quantity $G^{\le K}_n$ for comparisons in \Cref{lem:clt}.
\end{lemma}
\begin{proof}
    Note that because $\abs{K_{\phi,n}-\widetilde K_{\phi,n}}≤\abs\phi+1$ for every path $\phi$, we again have $\absnorm[\big]{\overline G^{\le K}_{\alpha,n} - \overline W^{\le K}_{\alpha,n}} \to 0$ and it suffices to show convergence of $\overline W^{\le K}_{\alpha,n}$. Now condition on $\mg[\big]{\widetilde S_{\phi,n}\mid \phi\in\mg{0,1}^*,n\in\NN}$. Then, the variables $\widetilde K_{\phi,n}$ for $\phi \in \mg{0,1}^\ast$ are hypergeometrically distributed and independent from each other. By \cite{hypergeometric} we have conditional convergence
    \begin{equation}\label{lem:swap:1}
        \kl*{
            \widetilde K_{\phi,n}
            -\frac{\widetilde S_{\phi0,n}\widetilde S_{\phi1,n}}{\widetilde S_{\phi,n}}
        }
        \kl*{\frac
            {{\widetilde S_{\phi,n}}^{1.5}}
            {\widetilde S_{\phi0,n}\widetilde S_{\phi1,n}}
        } \longdto Y_\phi
    \end{equation}
    with $Y_\phi\sim N(0,1)$, as long as for all but a finitely many $n$ both $\widetilde S_{\phi0,n}>0$ and $\widetilde S_{\phi1,n}>0$ hold and if
    ${{\widetilde S_{\phi,n}}^{1.5}}/
    {\widetilde S_{\phi0,n}\widetilde S_{\phi1,n}} \to 0$.
    These conditions holds almost surely, so \eqref{lem:swap:1} also holds almost surely and thus also unconditionally.
    By the strong law of large numbers and Slutzky's lemma, we obtain
    \begin{equation}\label{swap:clt:prf:k}
        \kl*{\widetilde K_{\phi,n}-\frac{\widetilde S_{\phi0,n}\widetilde S_{\phi1,n}}{\widetilde S_{\phi,n}}}\frac1{\sqrt n} \longdto Y_\phi\frac{I_{\phi0}I_{\phi1}}{{I_{\phi}}^{1.5}}.
    \end{equation}
    These $Y_\phi$ are independent of each other, the interval sizes and $(\widetilde S_\phi)_{\phi\in\mg{0,1}^\ast}$. The next component is the influence of variations in $\widetilde S_{\phi,n}$. For this, define the centred sums $\widehat S_{\phi,n} := \widetilde S_{\phi,n} - nI_{\phi,n}$. By the law of the iterated logarithm, we have $\widehat S_{\phi,n} = O\kl*{\sqrt{n\log\log n}}$ a.s. We expand
    \begin{align}\label{swap:clt:prf:s}
        \frac{\widetilde S_{\phi0,n}\widetilde S_{\phi1,n}}
        {\widetilde S_{\phi,n}}
        &= \frac{
            \kl{nI_{\phi0} + \widehat S_{\phi0,n}}
            \kl{nI_{\phi1} + \widehat S_{\phi1,n}}
        }{
            \kl{nI_{\phi}  + \widehat S_{\phi,n}}}
        \nl
        = n\frac{I_{\phi0}I_{\phi1}}{I_{\phi}}
        + \widehat S_{\phi0,n}\frac{I_{\phi1}}{I_{\phi}}
        + \widehat S_{\phi1,n}\frac{I_{\phi0}}{I_{\phi}}
        - \widehat S_{\phi,n}\frac{I_{\phi0}I_{\phi1}}{I_{\phi}^2}
        + o(\sqrt n)
    \end{align}
    almost surely using the Taylor expansion $\frac1{a+x} = \frac1a - \frac x{a²} + \O(x^2)$. By the multivariate central limit theorem, the tuple $\kl{n^{-0.5}\widehat S_\phi \mid \abs\phi ≤ k+1}$ conditioned on $\mathcal F_\infty$ converges in distribution to $\kl{Z_\phi \mid \abs\phi ≤ k+1}$. Putting \eqref{swap:clt:prf:k} and \eqref{swap:clt:prf:s} together, we obtain the convergence
    \begin{equation}\label{swap:clt:1}
        \overline W_{\phi,n} =
        \sqrt{n}\kl*{\widetilde K_{\phi,n} - n\frac{I_{\phi0}I_{\phi1}}{I_{\phi}}}
        \longdto
        Y_\phi\frac{I_{\phi0}I_{\phi1}}{I_{\phi}^{1.5}}
        + Z_{\phi0}\frac{I_{\phi1}}{I_{\phi}}
        + Z_{\phi1}\frac{I_{\phi0}}{I_{\phi}}
        - Z_{\phi}\frac{I_{\phi0}I_{\phi1}}{I_{\phi}^2}
    \end{equation}
    for all $\abs\phi ≤ k$ jointly. Since $\overline W_{n}^{≤K} = \sum_{\abs\phi≤K} \overline W_{\phi,n}\eins_{[L_\phi,R_\phi)}$, it thus converges towards a mixed centred Gaussian process formed by the right-hand side of \eqref{swap:clt:1}. The convergence is jointly with the quantity for comparisons since  we use the same limits $Z_\phi$.
\end{proof}

Since $\Var(Z_{\phi}\mid \mathcal F_\infty) = I_{\phi}-I_{\phi}^2$, the variance of the random variable in \eqref{swap:clt:1} is falling geometrically with $\abs\phi$ by \Cref{lem:k1} and the sum of these terms along some value $\alpha\in[0,1]$ converges to a centred normally distributed random variable.
$Z_\phi$ and $Y_\phi$ have the same distribution as described in \Cref{satz:swap}, and the sum of the terms in \eqref{swap:clt:1} coincides with the definition of our limit process $G_{\alpha}\swap$ from \eqref{defi:gswap}.

For tightness, we need an equivalent to \Cref{lem:lvlCosts}:
\begin{lemma}\label{lem:swap:lvlCosts}
    For the number of swaps in Hoare's partition, there exist constants $a,b>1$ and $C_b>0$ such that one can decompose
    for all $k,n\in\NN$
    \begin{equation}
        \Pk*{\exists\abs\phi = k:\: \abs{\overline W_{\phi,n}}>a^{-k} + n^{-0.5},\, K_1≤k} \le C_b\exp\kl*{-b^k} + c(k,n)
    \end{equation}
    where the double sequence $c(k,n)$ satisfies
    \begin{equation}\label{ckn_konvergiert}
        \sum_{k=1}^{\floor{4.5\log n}} c(k,n) \to 0.
    \end{equation}
\end{lemma}
\begin{proof}
    We fix some path $\abs{\phi} = k$ and condition on the interval sizes $I_{\phi 0}$ and $I_{\phi 1}$. Note that the distribution of $\widetilde K_{\phi, n}$ is invariant under permutations of the elements, so we can, without loss of generality, assume that the partition algorithm processes the elements in their original order.

    For any $1 \le t \le n$, consider how many elements the indices $i$ and $j$ from Hoare's partition algorithm have encountered before reaching the element $U_t$.
    The number $\widetilde S_{\phi1, t}$ of elements larger than $U_\phi$ left of $t$ is binomially $B\kl[\big]{t,I_{\phi1}}$-distributed and independent of the number $\widetilde S_{\phi0, n-t}$ of elements smaller than $U_\phi$ right of $t$, which is
    $B\kl[\big]{\widetilde S_\phi-t,I_{\phi0}}$-distributed. Now if both quantities are smaller than some $x\in \NN$, then both indices reach $U_t$ before swapping $x$ elements. When on the other hand both $\widetilde S_{\phi1, t}$ and $\widetilde S_{\phi0, n-t}$ are at least $x$, then $i$ and $j$  will swap at least $x$ times before reaching $U_t$. Hence, for all $x\in \RR$  we can bound, where the symbol $\lessgtr$ can be replaced by either $<$  (less than) at all three occurrences or by $>$ (larger than),
    \begin{equation}
        \Pk*{\widetilde K_{\phi,n}\lessgtr x} ≤ \Pk*{\widetilde S_{\phi1,t}\lessgtr x}\Pk*{\widetilde S_{\phi0,n-t}\lessgtr x}.
    \end{equation}
    Choosing $t=\floor{n\tfrac{I_{\phi0}}{I_{\phi}}} = n-\ceil{n\tfrac{I_{\phi0}}{I_{\phi}}}$, we obtain
    \begin{align}
        \Pk*{\abs{\overline W_{\phi,n}} \ge a^{-k} + n^{-0.5}}
        &= \Pk*{\abs*{\widetilde K_{\phi,n}-n\frac{I_{\phi0}I_{\phi1}}{I_\phi}}≥ \sqrt n a^{-k} + 1}
        \nl≤ \Pk*{\abs*{\widetilde S_{\phi1,t}
          -n\frac{I_{\phi0}I_{\phi1}}{I_\phi}}
          ≥ \sqrt n a^{-k} +1}
        \nl\phantom=\times\Pk*{\abs*{\widetilde S_{\phi0,n-t}
          -n\frac{I_{\phi0}I_{\phi1}}{I_\phi}}
          ≥ \sqrt n a^{-k} + 1}.
    \end{align}
    The latter difference is not centred because of rounding, but at most 1 of from centring. This offsets with the additional $n^{-0.5}$, and by Chernoff bounds we get
    \begin{align}
        \Pk*{\abs{\overline W_{\phi,n}} ≥ a^{-k} + \frac1{\sqrt n}}
        &≤ \exp\kl*{
            - \frac{\floor{n\frac{I_{\phi0}}{I_{\phi}}}a^{-2k}}{n(2I_{\phi1} + a^{-k}n^{-0.5})}
            - \frac{\ceil{n\frac{I_{\phi1}}{I_{\phi}}}a^{-2k}}{n(2I_{\phi0} + a^{-k}n^{-0.5})}
        }
        \nl≤ \exp\kl*{
            - \frac{a^{-2k}}{2I_\phi + a^{-k}n^{-0.5}}
        }.
    \end{align}
    This is the same bound as in \eqref{prf:chernoff} in the proof of \Cref{lem:lvlCosts} and the remaining of the present proof can then be done analogously.
\end{proof}
\begin{proof}[Proof of \Cref{satz:swap}]
By the same proof as for \Cref{prop:rest-klein}, we conclude that $\absnorm{G\swap _n - \overline G^{≤K}_n}$ converges to 0 in probability. This guarantees tightness, and with the fdd convergence in \Cref{lem:swap:clt} this implies process convergence.
\end{proof}

\subsection{$\eps$-tame cost functions} \label{subsec:proof:bits}
In the present section we present a proof for Theorem \ref{thm:beta}.

First note that the expression in \eqref{cov:beta} is well-defined in view of the following lemma using $\eps$-tameness, cf.~\cite[Lemma 3.1]{fima14}:
\begin{lemma}\label{lem:beta:norm}
    Let $[L,R) \subseteq [0,1]$ be an interval of length $I=R-L>0$ and $u\in [L,R)$.
    For  $V$ being $\unif[0,1]$ distributed set $X := \eins_{[L,R)}(V) \beta(u,V)$. Then, for every $s \in (0,\eps\inv)$, uniformly in $L,R$,  we have
    \begin{equation*}
        \Ek{X^s} = I\cdot\Ek*{X^s\given V\in [L,R)}
        = \O\kl[\Big]{I^{1-\eps s}}.
    \end{equation*}
\end{lemma}
\begin{proof}
    The tail distribution $F(x) = \P(X≥x)$ for $x>0$ is bounded by both $\Pk{V\in [L,R)} =  I$ and $\Pk{\beta(u,V)> x} ≤ C_\beta x^{-\eps\inv}$ for some constant $C$. These two bounds are equal when $x = \kl{I/C_\beta}^{-\eps}$, so
    \begin{align}
        \Ek{X^s} = \int_0^\infty sx^{s-1} F(x) \mathrm dx
        & ≤ I \int_0^{\kl{I/C_\beta}^{-\eps}} sx^{s-1}\mathrm dx
        + C_\beta \int_{\kl{I/C_\beta}^{-\eps}}^\infty sx^{s-1-\eps\inv} \mathrm dx
        \nl = O\kl*{I^{1-s\eps}} + O\kl*{I^{-s\eps + 1}},
    \end{align}
    where $s\eps < 1$ is required for the second integral to exist.
\end{proof}

We use the same approach as for the other cost measures: First, we show that the costs at some level are close to a sum of independent random variables to which we apply the central limit theorem. Then we show that the residual below some level is converging to $0$.
Since our cost now depends on the actual  values $U_1,\ldots,U_n$ instead of just whether the $U_i$ are contained in the respective interval or not, we cannot define $\widetilde S_{\alpha,k,n}$ as before. Instead, we only replace $U_i$ with $\widetilde U_i$ if it is below some level $K\in\NN$, defining for $i\in\NN$
\begin{equation}
    \widetilde U^{(K)}_i := \begin{cases}
        U_i & \abs{\phi_i} > K, \\
        \widetilde U_i & \abs{\phi_i} ≤ K
    \end{cases}
\end{equation}
and then for $k≤K$
\begin{align}
    S_{\alpha,k,n}^\beta &:=
    \sum_{i=\tau_{\alpha,k}+1}^n \eins_{[L_{\alpha,k},R_{\alpha,k})}
    \kl*{U_i}
    \beta\kl*{U_{\tau_{\alpha,k}},U_i}
    \\
    \widetilde S^{\beta,K}_{\alpha,k,n} &:=
    \sum_{i=1}^n \eins_{[L_{\alpha,k},R_{\alpha,k})}
    \kl*{\widetilde U^{(K)}_i}
    \beta\kl*{U_{\tau_{\alpha,k}},\widetilde U^{(K)}_i}
    \\
    I_{\alpha,k}^\beta &:=
     \int_0^1 \eins_{[L_{\alpha,k},R_{\alpha,k})}
    \kl*{v}
    \beta\kl*{U_{\tau_{\alpha,k}},v}
    \mathrm dv
.
\end{align}
Note that, conditional on $\mathcal F_K$, the values
$\widetilde U^{(K)}_1,\widetilde U^{(K)}_2,\dots$ are i.i.d.\ unif[0,1] and hence
$\widetilde S^{(K)}_{\alpha,k,n}$ is a sum of i.i.d.\ random variables distributed as $X^\beta_{\alpha,k}$, with expectation $I_{\alpha,k}^\beta$. With these random variables, we define similar to before
\begin{align}
    G_{\alpha,k,n}^\beta &:=
    \frac{S_{\alpha,k,n}^\beta - nI_{\alpha,k,n}^\beta}{\sqrt n},
    &
    W_{\alpha,k,n}^{\beta,K} &:=
    \frac{\widetilde S_{\alpha,k,n}^{\beta,K} - nI_{\alpha,k,n}^\beta}{\sqrt n}.
\end{align}
The $G_{\alpha,k,n}^\beta$ are the single steps within the residual process, and $W_{\alpha,k,n}^{\beta,K}$ are the normalized sum of i.i.d.\ random variables used to approximate them.
\begin{lemma}\label{lem:beta:clt}
    The residuals $\sum_{k=0}^K G_{\alpha,k,n}$ above some level $K\in\NN$ converge conditional on $\mathcal F_K$ to a mixture of centred Gaussian processes with random covariance function given by
    \begin{equation}\label{cov:beta:k}
        (\alpha,\gamma) \mapsto
        \Cov\kl*{
            \sum_{k=0}^K X^\beta_{\alpha,k},
            \sum_{k=0}^K X^\beta_{\gamma,k}
            \mid \mathcal F_K},
    \end{equation}
    with $X^\beta_{\alpha,k}$ as in \eqref{def:xbeta}.
\end{lemma}

\begin{proof}
    Conditioned on $\mathcal F_K$,  $\sum_{k=0}^K W^{\beta,K}_{\alpha,k,n}$ is a step function with $2^K$ different values. These values are each sums of $n$ independent, $\sum_{k=0}^K X^\beta_{\alpha,k}$-distributed variables. By the multivariate central limit theorem $\sum_{k=0}^K W^{\beta,K}_{\alpha,k,n}$ thus converges to a mixed Gaussian process with covariances \eqref{cov:beta:k}.
    The difference between $\sum_{k=0}^K W^{\beta,K}_{\alpha,k,n}$ and $\sum_{k=0}^K G^{\beta}_{\alpha,k,n}$ is at most the different costs of $U_\phi$ and $\widetilde U_\phi$ at paths $\phi\in\mg{0,1}^\ast$, $\abs\phi ≤ k$. Those are finitely many and are then divided by $\sqrt n$. By Slutzky's lemma, $\sum_{k=0}^K G^{\beta}_{\alpha,k,n}$ hence also converges to the mixed Gaussian process.
\end{proof}

With the same technique as before (see \Cref{lem:fdd} and the proof of \Cref{mainTheorem}), we can prove \Cref{thm:beta} with the following proposition:

\begin{proposition}\label{prop:beta:prop:rest-klein}
    For all $0<\eps<\frac14$ the residuals below level $K \in \NN$ converge
    as $n$, $K\to\infty$  to $0$:
    \begin{equation}
        \sup_{\alpha\in[0,1]}\abs[\Big]{\sum_{k>K}G_{\alpha,k,n}}\overset\P\longrightarrow 0
    \end{equation}
\end{proposition}

\noindent This is a consequence of the following lemma:
\begin{lemma}\label{lem:beta:snorm}
    For $s≥2$, we have uniformly in $\alpha,k,n, I_{\alpha,k}$
    \begin{equation}\label{beta:snorm}
        \Ek*{\abs[\big]{G^\beta_{\alpha,k,n}}^s \given L_{\alpha,k}, R_{\alpha,k}}
        = O\kl*{
            I_{\alpha,k}^{2-s\eps}n^{2-\frac s2}
            + k^{\frac s2}I_{\alpha,k}^{(\frac12-\eps)s}
        }.
    \end{equation}
\end{lemma}

\begin{proof}
    First, let $s\in 2\NN^+$ be even. Then decompose $G^\beta_{\alpha,k,n}$ into
    \begin{align}
        \kl[\big]{G^\beta_{\alpha,k,n}}^s
        &≤  n^{-\frac s2}\kl*{
            S^\beta_{\alpha,k,n} - (n-\tau_{\alpha,k})^+ I^\beta_{\alpha,k}
            - (\tau_{\alpha,k}\minv n)I^\beta_{\alpha,k}
        }^s
        \nl ≤ 2^s n^{-\frac s2}\kl*{
            \kl*{S^\beta_{\alpha,k,n} - (n-\tau_{\alpha,k})^+ I^\beta_{\alpha,k}}^s
            - (\tau_{\alpha,k}\minv n)^s (I^\beta_{\alpha,k})^s
        }.\label{sum_72}
    \end{align}
    The first part is a centred sum of $(n-\tau_{\alpha,k})^+$ independent, $X^\beta_{\alpha,k}$-distributed random variables; the second part arises from $n-(n-\tau_{\alpha,k})^+ = n \minv \tau_{\alpha,k}$. If we condition on $\tau_k$, expand the sum in $S^\beta_{\alpha,k,n}$ and group the resulting terms by the number of occurrences  of  $U_i$, then terms where  $U_i$ appears only once vanish caused by centring, and we have
    \begin{align}
        \MoveEqLeft
        \Ek*{\kl*{S^\beta_{\alpha,k,n} - (n-\tau_{\alpha,k})^+I^\beta_{\alpha,k}}^s
            \given L_{\alpha,k}, R_{\alpha,k},\tau_{\alpha,k}}
        \nl= \Ek[\bigg]{\kl[\Big]{
            \sum_{i=\tau_{\alpha,k}+1}^n\kl[\big]{
                \eins_{[L_{\alpha,k},R_{\alpha,k})}\kl*{U_i}\beta\kl*{U_{\tau_{\alpha,k}},U_i}
                - I^\beta_{\alpha,k}
            }}^s
            \given L_{\alpha,k}, R_{\alpha,k},\tau_{\alpha,k}}
        \nl= \sum_{l=1}^{s/2} \binom{(n-\tau_{\alpha,k})^+}l
        \sum_{i_1+\dots+i_l=s \atop i_1,\dots, i_l ≥ 2}
        \prod_{j=1}^l
        \Ek*{\kl{X^\beta_{\alpha,k}-I_{\alpha,k}^\beta}^{i_j}\given L_{\alpha,k}, R_{\alpha,k}}.\label{eq_old73}
    \end{align}
    By \Cref{lem:beta:norm}, $I^\beta_{\alpha,k} = O(I_{\alpha,k}^{1-\eps})$ and $\Ek{(X^\beta_{\alpha,k})^i} = O(I_{\alpha,k}^{1-\eps i})$, so the last product is of order
    $O(I_{\alpha,k}^{l-\eps s})$, and the sum in (\ref{eq_old73}) is equal to
    \begin{equation}
        = \sum_{l=1}^{s/2} O\kl[\big]{n^l I^{l-\eps s}_{\alpha,k}}
        = O(nI^{1-\eps s}_{\alpha,k})
        + O\kl[\Big]{n^{s/2}I^{(\frac12-\eps)s}_{\alpha,k}}.
    \end{equation}
    This holds uniformly over all $\tau_{\alpha,k}$. When $\tau_{\alpha,k}≥n$, the latter sum vanishes, so by multiplying with $\Pk{\tau_{\alpha,k}≤n \given I_{\alpha,k}} ≤ 1 \minv I_{\alpha,k}n$ we obtain
    \begin{eqnarray}
        n^{-\frac s2}\Ek{S^\beta_{\alpha,k,n} - (n-\tau_{\alpha,k})^+ \given L_{\alpha,k}, R_{\alpha,k}}
        = O\kl[\Big]{I^{2-\eps s}_{\alpha,k}n^{2-\frac s2}
        + I^{(\frac12-\eps)s}_{\alpha,k}},
    \end{eqnarray}
    which has the order stated in \eqref{beta:snorm}.

    To bound the second summand appearing in (\ref{sum_72}) we use an idea from \cite{Matterer,fill:quickselect-residual}: The estimate $\tau_{\alpha,k}\minv n ≤ n^{\frac12}\tau_{\alpha,k}^{\frac 12}$ implies
    \begin{align}
        n^{-\frac s2}(\tau_{\alpha,k}\minv n)^s
        \kl*{I^\beta_{\alpha,k}}^s
        &=
        \tau_{\alpha,k}^{\frac s2}
        \kl*{I^\beta_{\alpha,k}}^s.
    \end{align}
    Conditional on $\mathcal F_k$, $\tau_{\alpha,k}$ is the sum of $k$ independent geometrically distributed variables with success probabilities $I_{\alpha,0}, \dots, I_{\alpha,k-1}$ and thus smaller in probability than a sum of $k$ independent Geom($I_{\alpha,k}$)-distributed random variables $\Gamma_1,\dots, \Gamma_k$. Thus
    \begin{align}
        \Ek*{
            \tau_{\alpha,k}^{\frac s2}
        \given L_{\alpha,k},R_{\alpha,k}}
        ≤
        \Ek[\bigg]{
        \kl[\Big]{\sum_{j=1}^k \Gamma_j}^{\frac s2}
        \given I_{\alpha,k}
        }
        ≤ k^{\frac s2} \Ek[\big]{\Gamma_1^{\frac s2}\given I_{\alpha,k}}
        = O(k^{s/2}I_{\alpha,k}^{-s/2}).
    \end{align}
    Together with $\kl*{I^\beta_{\alpha,k}}^s = O(I^{(1-\eps)s}_{\alpha,k})$, this implies
    \begin{equation}
         \Ek*{
            n^{-\frac s2}(\tau_{\alpha,k}\minv n)^s
            \kl*{I^\beta_{\alpha,k}}^s
        \given L_{\alpha,k},R_{\alpha,k}}
            = O\kl*{k^{s/2}I_{\alpha,k}^{(\frac12-\eps)s}}.
    \end{equation}
    This completes the proof of the present lemma for $s$ even. The general case $s≥2$ is implied by Lyapunov's interpolation inequality.
\end{proof}
We are now prepared to prove \Cref{prop:beta:prop:rest-klein}:
\begin{proof}[Proof of \Cref{prop:beta:prop:rest-klein}]
    For some path $\phi \in \mg{0,1}^k$ the interval length $I_{\phi}$ is distributed as the product of $k$ i.i.d.\ unif[0,1] random variables, so
    \begin{equation}
        \Ek*{\abs[\big]{G^\beta_{\phi,n}}^s \given L_{\phi}, R_{\phi}}
        = O\kl*{
            I_{\phi}^{2-s\eps}n^{2-\frac s2}
            + k^{\frac s2}I_{\phi}^{(\frac12-\eps)s}
        }.
    \end{equation}
 \Cref{lem:beta:snorm} hence implies
    \begin{equation}
        \Ek*{\abs[\big]{G^\beta_{\phi,n}}^s}
        = O\kl*{
            \fracc1{3-s\eps}^k n^{2-\frac s2}
            + k^{\frac s2}\fracc1{1+(\frac12-\eps)s}^k
        }.
    \end{equation}
    Summing over all $\phi$ with $\abs\phi=k$ implies
    \begin{equation}
        \Ek*{\max_{\abs\phi=k}\abs[\big]{G^\beta_{\phi,n}}^s}
        = O\kl*{
            \fracc2{3-s\eps}^k n^{2-\frac s2}
            + k^{\frac s2}\fracc2{1+(\frac12-\eps)s}^k
        }
    \end{equation}
    and the $L_s$-norm is thus bounded by
    \begin{equation}
        \lnorm*s{\max_{\abs\phi=k}\abs[\big]{G^\beta_{\phi,n}}}
        = O\kl*{
            \fracc2{3-s\eps}^{\frac ks} n^{\frac 2s-\frac12}
            + k^{\frac 12}\fracc2{1+(\frac12-\eps)s}^{\frac ks}
        }.
    \end{equation}
    Summing over $k\in\N$ yields a convergent series if $s<\eps\inv, s≥4$, and $s≥\frac{2}{1-2\eps}$. If $\eps < \frac14$, then $s$ can be chosen as $4$, and
    \begin{equation}
        \lnorm[\bigg]s{\max_{\abs\phi=k}\abs[\Big]{\sum_{k≥K}G^\beta_{\phi,n}}}
        ≤ \sum_{k≥K}\lnorm*s{\max_{\abs\phi=k}\abs[\big]{G^\beta_{\phi,n}}} \to 0.
    \end{equation}
\end{proof}

\Cref{thm:beta} now follows from \Cref{prop:beta:prop:rest-klein} and \Cref{lem:beta:clt} as for the previously discussed complexity measures.

\subsection{Path properties of the limit process}\label{sec:paths}
In the present section we prove the claims from Section \ref{sec:path} on path properties of the limit processes.
\begin{lemma}\label{lem:d2dg-equivalent}
    The metrics $d_2$ and $d_G$ are a.s.\ equivalent.
\end{lemma}
\begin{proof}
    To rewrite \eqref{def:dg} in terms of the interval lengths, recall that the covariances can be expressed with a unif[0,1] variable $V$ as in \eqref{cov_desc},
\begin{align}
    d_G(\alpha,\beta)^2 &=
     \Ek[\big]{\kl{J(\alpha,V)-J(\beta,V)}^2 \given \mathcal F_\infty} - (S_\alpha - S_\beta)^2
     \nl = \sum_{l≥1} (2l-1) \Pk[\big]{\abs{J(\alpha,V)-J(\beta,V)} ≥ l \given \mathcal F_\infty} - (S_\alpha - S_\beta)^2.
\end{align}
When $J(\alpha,V)≥l+J(\beta,V)$, the value $V$ must have followed the path of $\alpha$ for at least $l$ steps more than it has followed  the path of $\beta$. That happens if and only if $V$ is in the interval $I_{\alpha, J(\alpha,\beta)+l}$. Hence, writing $J:=J(\alpha,\beta)$,
\begin{align}
    d_G(\alpha,\beta)^2
     &= \sum_{k>J} (2(k-J)-1)\kl{I_{\alpha,k} + I_{\beta,k}}
      - \kl[\bigg]{\sum_{k>J} \kl{I_{\alpha,k} - I_{\beta,k}}}^2
     \nl= \sum_{k>J} \kl[\bigg]{\kl{I_{\alpha,k} + I_{\beta,k}}
     (2(k-J)-1) - \abs{I_{\alpha,k} - I_{\beta,k}}\sum_{l>J} \abs{I_{\alpha,l}-I_{\beta,l}}}
     .
     \label{di-formula}
\end{align}
Since by \Cref{lem:k1} the sums converge a.s.\ uniformly for $J\to\infty$, every $d_G$ ball contains a $d_2$ ball. Conversely, by reordering the last sum such that $l≤k$, and bounding $\abs{I_{\alpha,l}-I_{\beta,l}} < 1$,
\begin{align}
    d_G(\alpha,\beta)^2 &≥ \sum_{k>J} (2(k-J)-1)\br[\big]{\kl{I_{\alpha,k} + I_{\beta,k}} - \abs{I_{\alpha,k} - I_{\beta,k}}}
    \nl = \sum_{k>J} 2(2(k-J)-1)\kl{I_{\alpha,k} \minv I_{\beta,k}}.
    \label{di-bound}
\end{align}
Since for all $\alpha$,$\beta$ and $k$ we have $I_{\alpha,k},I_{\beta,k}>0$ a.s., every $d_2$ ball thus also contains a $d_G$ ball and the metrics are a.s.\ equivalent.
\end{proof}

To study H\"older continuity we prepare with bounds on the limit processes in terms of interval length of the underlying partition, cf.~(\ref{defi:irl}).

\begin{lemma}\label{lem:gphi-small}
Almost surely for all $\phi\in\mg{0,1}^\ast$ we have the following bounds:
    \begin{enumerate}
        \item\label{lem:cont:1} $G\comp_{\phi, \infty} :=Z_\phi = \O\kl*{\sqrt{\abs\phi I_\phi}}$,
        \item\label{lem:cont:2} $G_{\phi, \infty}\swap := Y_\phi\frac{I_{\phi0}I_{\phi1}}{I_{\phi}^{3/2}}
    + Z_{\phi0}\frac{I_{\phi1}}{I_{\phi}}
    + Z_{\phi1}\frac{I_{\phi0}}{I_{\phi}}
    - Z_{\phi}\frac{I_{\phi0}I_{\phi1}}{I_{\phi}^2} = \O\kl*{\sqrt{\abs\phi I_\phi}}$,
        \item\label{lem:cont:3} $G^{\beta}_{\phi, \infty} = \O\kl*{\sqrt{\abs\phi I_\phi^{1-2\eps}}}$.
    \end{enumerate}
\end{lemma}
In particular, we can subsequently only focus on $G^\beta_{\phi,\infty}$, since the H\"older continuity for comparisons and swaps from Theorem \ref{satz:hölder} is then covered by choosing $\eps=0$.

\begin{proof}
We already showed part (\ref{lem:cont:1}) concerning $G\comp_{\phi, \infty}$   in \Cref{lem:zphi-small}. For part (\ref{lem:cont:2}) note that the factors in front of $Z_\phi$, $Z_{\phi0}$ and $Z_ {\phi1}$ are smaller than $1$, so we only require an additional bound for $Y_\phi$. For a sequence $(a_k)$ in $\RR^+$,
\begin{equation}
    \sum_{\abs\phi = k} \Pk{\abs{Y_\phi} > a_k} ≤ \sum_{\abs\phi = k} 2\exp\kl*{-\frac{a_k^2}2} = 2\exp\kl*{-\frac{a_k^2}2 + k\log(2)}.
\end{equation}
Choosing $a_k=2\sqrt k$ is sufficient for the sequence of the latter terms to converge, thus $Y_\phi = \O\kl{\sqrt{\abs \phi}}$ from the Borel--Cantelli lemma. With $I_{\phi0}I_{\phi1}I_{\phi}^{-1.5} ≤ \sqrt{I_{\phi}}$, this concludes the proof of  part (\ref{lem:cont:2}). For part (\ref{lem:cont:3}), recall that $G^\beta_{\phi,\infty}$ is mixed Gaussian with the same variance as $X^\beta_\phi$, as defined in \eqref{cov:beta}. By \Cref{lem:beta:norm} we have $\Ek[\big]{\kl{X^\beta_\phi}^s} = \O\kl{I_\phi^{1-\eps s}}$, hence $\Var(X^\beta_\phi) = \O\kl{I_\phi^{1-2\eps}}$.
We may repeat the same argument as for $Z_\phi$ and $Y_\phi$ to get $G^\beta_{\phi,\infty} = \O\kl*{\sqrt{\abs\phi I_\phi^{1-2\eps}}}$.
\end{proof}

\begin{proof}[Proof of \Cref{satz:hölder} for $d_2$] We prove the statement of \Cref{satz:hölder} only for $G^\beta$ with $\eps$-tame $\beta$. Note that in view of \Cref{lem:k1} the same proof applies to $G\comp$ and $G\swap$ with $\eps=0$.
\Cref{lem:gphi-small} implies that for any $\alpha,\beta\in[0,1]$ we have almost surely
\begin{equation}\label{g-bound}
    \abs[\big]{G^\beta_{\alpha,\infty}-G^\beta_{\beta,\infty}} ≤ \sum_{k>J} \abs{G^\beta_{\alpha,k,\infty}}+\abs{G^\beta_{\beta,k,\infty}}
    ≤ \sum_{k>J} \sqrt k \O\kl*{I_{\alpha,k}^{0.5-\eps} + I_{\beta,k}^{0.5-\eps}}.
\end{equation}
By \Cref{lem:k1}, these sizes are further bounded by
\begin{align}
    \abs[\big]{G^\beta_{\alpha,\infty}-G^\beta_{\beta,\infty}}
     &≤ \sum_{k>J}k^{1-\eps} \O\kl[\bigg]{\fracc23^{(1 - 2\eps)k/4}}
     = \O\kl*{J^{1-\eps}2^{-J\kl*{\frac14\log_2\frac32}(1-2\eps)}},
     \nonumber\\
     \log_2 \abs[\big]{G^\beta_{\alpha,\infty}-G^\beta_{\beta,\infty}} &≤ \kl[\big]{1+\mathrm o(1)}\log_2\kl*{d_2(\alpha,\beta)}\kl*{\frac14\log_2\frac32}(1-2\eps).
\end{align}
This implies $d_2$-Hölder continuity with all exponents smaller than $\frac14\log_2\frac32(1-\eps)$.
\end{proof}

Since in \eqref{di-formula} the interval sizes are decreasing exponentially, we see that $d_G(\alpha,\beta)^2$ is dominated by the term for $k=J+1$, which is $I_{\alpha,J}$. On the other hand $(G_{\alpha,\infty}-G_{\beta,\infty})^2$ is approximately $\sum_{k>J}I_{\alpha,k} + I_{\beta,k}$. So we may bound the intervals $I_{\alpha,k}, k>J$, with the length of their ancestor $I_{\alpha, J}$, again using a Borel--Cantelli argument.

\begin{lemma}\label{lem:relative-geometric}
    For any $\delta\in(0,1)$ and $s>\frac1\delta$ we have almost surely for all but finitely many $\phi,\psi\in\mg{0,1}^\ast$ that
    \begin{equation}\label{eq:relative-geometric}
        \frac{I_{\phi\psi}}{I_{\phi}^{1-\delta}} ≤ \abs\psi\fracc2{1+s}^{{\abs\psi}/s}.
    \end{equation}
    In particular, for $\gamma>0$, the following holds uniformly in $\alpha$:
    \begin{equation}
        \sum_{k>J} \sqrt k I_{\alpha,k}^\gamma
        = \sum_{k>J} I_{\alpha,k}^{\gamma(1-\delta)} k^{1.5}\fracc2{1+s}^{\gamma k/s}
        = \O\kl*{J^{1.5} I_{\alpha,J}^{\gamma(1-\delta)}}.
    \end{equation}
\end{lemma}
\begin{proof}
    Note that $I_{\phi\psi}/I_{\phi}$ is the product of $\abs\psi$ i.i.d.\ unif[0,1] random variables and independent of $I_{\phi}$, which is also a product of $\abs\phi$ i.i.d.\ unif[0,1] random variables. By Markov bounding,
    \begin{align}
        \Pk*{\fracc{I_{\phi\psi}}{I_{\phi}^{1-\delta}}^s ≤ \abs\psi^s\fracc2{1+s}^{\abs\psi}}
        &≤ \Ek*{\fracc{I_{\phi\psi}}{I_{\phi}}^sI_\phi^{s\delta}}
        \abs\psi^{-s}\fracc{1+s}2^{\abs\psi}
        \nl= \fracc1{1+s}^{\abs\psi}\fracc1{1+s\delta}^{\abs\phi} \abs\psi^{-s}\fracc{1+s}2^{\abs\psi}
        \nl= 2^{-\abs\psi-\abs\phi} \fracc{2}{1+s\delta}^{\abs\phi} \abs\psi^{-s}.
    \end{align}
    Since $s\delta >1$ and $s>1$, the sum of these probabilities over all $\phi$ and $\psi$ converges. Hence, as in previous arguments, the Borel--Cantelli lemma implies the assertion.
\end{proof}

\begin{proof}[Proof of \Cref{satz:hölder} for $d_G$]
    With \eqref{g-bound} and \Cref{lem:relative-geometric}, we obtain
    \begin{equation}\label{g-geometric}
        \abs{G_{\alpha,\infty}-G_{\beta,\infty}}
        = \O\kl*{\sum_{k>J} \sqrt k
            \kl*{{\abs{I_{\alpha,k}}}^{0.5-\eps} +{\abs{I_{\beta,k}}}^{0.5-\eps}}}
        = \O\kl*{J^{1.5} I_{\alpha,J}^{(0.5-\eps)(1-\delta)}}
    \end{equation}
    for every $\delta \in (0,1)$ almost surely and uniformly in $\alpha$ and $\beta$.
    By the same reasoning as used to derive \eqref{di-bound} from \eqref{di-formula}, we can bound $d_G$ with only the term arising from $k=l=J+1$,
    \begin{equation}\label{di-geometric}
        d_G(\alpha,\beta)^2 ≥ \kl*{I_{\alpha,J+1} + I_{\beta,J+1}} - \kl*{I_{\alpha,J+1}-I_{\beta,J+1}}^2 ≥ I_{\alpha,J} - I_{\alpha,J}^2 \sim I_{\alpha,J}.
    \end{equation}
   Now, \eqref{g-geometric} and \eqref{di-geometric} imply Hölder continuity with any exponent smaller than $1-2\eps$.
\end{proof}





\providecommand{\bysame}{\leavevmode\hbox to3em{\hrulefill}\thinspace}
\providecommand{\MR}{\relax\ifhmode\unskip\space\fi MR }
\providecommand{\MRhref}[2]{%
  \href{http://www.ams.org/mathscinet-getitem?mr=#1}{#2}
}
\providecommand{\href}[2]{#2}

\end{document}